\documentclass{article}
\usepackage[margin=1.5in]{geometry}  

\usepackage[utf8]{inputenc} 
\usepackage[T1]{fontenc}    
\usepackage{hyperref}       
\usepackage{url}            
\usepackage{booktabs}       
\usepackage{amsfonts}       
\usepackage{nicefrac}       
\usepackage{microtype}      
\usepackage{xcolor}         

\usepackage{graphicx}
\usepackage{amsmath,amssymb}
\usepackage{array}
\usepackage{enumitem}
\usepackage{xcolor}
\usepackage{bm}
\usepackage{multirow}
\usepackage{mathdots}
\usepackage{caption, subcaption}
\usepackage{amsthm}
\usepackage{pifont}

\usepackage{algorithm}
\usepackage[noend]{algpseudocode}

\newtheorem{theorem}{Theorem}
\newtheorem{definition}{Definition}
\newtheorem{lemma}[theorem]{Lemma}
\newtheorem{proposition}[theorem]{Proposition}
\newtheorem{assumption}{Assumption}
\newtheorem{corollary}[theorem]{Corollary}
\newtheorem{remark}[theorem]{Remark}

\numberwithin{example}{section}
\numberwithin{theorem}{section}
\numberwithin{definition}{section}

\providecommand{\ie}		{\emph{i.e\@.}\xspace}
\providecommand{\eg}		{\emph{e.g\@.}\xspace}

\providecommand{\myurl}[1][]	{\texttt{web.eecs.umich.edu/$\sim$fessler#1}\xspace}

\providecommand{\onweb}[1]	{Available from \myurl.}

\long\def\comment#1{}

\providecommand{\bcent}		{\begin{center}}
\providecommand{\ecent}		{\end{center}}
\providecommand{\benum}		{\begin{enumerate}}
\providecommand{\eenum}		{\end{enumerate}}
\providecommand{\bitem}		{\begin{itemize}}
\providecommand{\eitem}		{\end{itemize}}
\providecommand{\bvers}		{\begin{verse}}
\providecommand{\evers}		{\end{verse}}
\providecommand{\btab}		{\begin{tabbing}}	
\providecommand{\etab}		{\end{tabbing}}

\newcounter{blist}
\providecommand{\blistmark}	{\makebox[0pt]{$\bullet$}}
\providecommand{\blistitemsep}	{0pt}
\providecommand{\blist}[1][]	{%
\begin{list}{\blistmark}{%
\usecounter{blist}%
\setlength{\itemsep}{\blistitemsep}%
\setlength{\parsep}{0pt}%
\setlength{\parskip}{0pt}%
\setlength{\partopsep}{0pt}%
\setlength{\topsep}{0pt}%
\setlength{\leftmargin}{1.2em}%
\setlength{\labelsep}{0.5\leftmargin}
\setlength{\labelwidth}{0em}%
#1}
}
\providecommand{\elist}		{\end{list}}

\providecommand{\blistitemsep}	{0pt}
\providecommand{\bjfenum}[1][]	{%
\begin{list}{\bcolor{\arabic{blist}.} }{%
\usecounter{blist}%
\setlength{\itemsep}{\blistitemsep}%
\setlength{\parsep}{0pt}%
\setlength{\parskip}{0pt}%
\setlength{\partopsep}{0pt}%
\setlength{\topsep}{0pt}%
\setlength{\leftmargin}{0.0em}%
\setlength{\labelsep}{1.0\leftmargin}
\setlength{\labelwidth}{0pt}%
#1}
}

\newcounter{blistAlph}
\providecommand{\blistAlph}[1][]
{\begin{list}{\makebox[0pt][l]{\Alph{blistAlph}.}}{%
\usecounter{blistAlph}%
\setlength{\itemsep}{0pt}\setlength{\parsep}{0pt}%
\setlength{\parskip}{0pt}\setlength{\partopsep}{0pt}%
\setlength{\topsep}{0pt}%
\setlength{\leftmargin}{1.2em}%
\setlength{\labelsep}{1.0\leftmargin}
\setlength{\labelwidth}{0.0\leftmargin}#1}%
}

\newcounter{blistRoman}
\providecommand{\blistRoman}[1][]
{\begin{list}{\Roman{blistRoman}.}{%
\usecounter{blistRoman}%
\setlength{\itemsep}{0.5em}\setlength{\parsep}{0pt}%
\setlength{\parskip}{0pt}\setlength{\partopsep}{0pt}%
\setlength{\topsep}{0pt}%
\setlength{\leftmargin}{4em}%
\setlength{\labelsep}{0.4\leftmargin}
\setlength{\labelwidth}{0.6\leftmargin}#1}%
}

\usepackage{bbm} 
\providecommand{\jfbbm}[1]	{\xmath{\mathbbm{#1}}} 
\providecommand{\qed}[1][0pt]	{\hfill\raisebox{#1}{\inmath{\Box}}} 

\providecommand{\reals}		{\jfbbm{R}}

\providecommand{\inprod}[2]	{\xmath{\mathop{\langle #1,\, #2 \rangle}\nolimits}}
\providecommand{\Inprod}[2]	{\xmath{\left\langle #1,\ #2 \right\rangle}}

\let\equivsave\equiv
\def\equiv{\xmath{\equivsave}}

\providecommand{\ba}[1]		{\left[ \begin{array}{#1}}
\providecommand{\ea}		{\end{array} \right]}
\providecommand{\be}		{\begin{equation}}
\providecommand{\ee}[1]		{\label{#1}\end{equation}}
\providecommand{\bea}		{\begin{eqnarray}}
\providecommand{\eea}[1]	{\label{#1}\end{eqnarray}}
\providecommand{\beas}		{\begin{eqnarray*}}
\providecommand{\eeas}		{\end{eqnarray*}}
\providecommand{\beals}[1][1]	{\begin{alignat*}{#1}}	
\providecommand{\eeals}		{\end{alignat*}}

\providecommand{\berr}[2]{
\bgroup
\renewcommand{\theequation}{#1}
\be
#2
\ee{e,#1}
\egroup
\ignorespaces
}

\providecommand{\bearr}[2]{
\bgroup
\renewcommand{\theequation}{#1}
\bea
#2
\eea{e,#1}
\egroup
\ignorespaces
}

\providecommand{\inmath}	{\ensuremath}
\providecommand{\xmath}[1]	{\inmath{#1}\xspace}
\providecommand{\bmath}[1]	{\xmath{\bm{#1}}}	

\providecommand{\paren}[1]	{\xmath{\left(#1\right)}}

\providecommand{\xfrac}[2]	{\xmath{\frac{#1}{#2}}}

\providecommand{\pfrac}[3][]	{\xmath{\!\paren{#1\frac{#2}{#3}}}}

\providecommand{\sinc}		{\Ofun{\mathrm{sinc}}}

\usepackage{jf-accent}
\usepackage{jf-fun}
\usepackage{jf-names}

\newcommand{\cmark}{\ding{51}}%
\newcommand{\cred} {\color{red}}
\definecolor{darkgreen}{rgb}{0.0, 0.4, 0.13}

\newcommand{\prox} {\operatorname{prox}}
\newcommand{\inter} {\operatorname{int}}
\newcommand{\dom} {\operatorname{dom}}
\newcommand{\ran} {\operatorname{ran}}

\newcommand{\gra} {\operatorname{gra}}

\newcommand{\Us} {\mathcal{U}}
\newcommand{\Vs} {\mathcal{V}}

\newcommand{\zero} {\bmath{0}}

\newcommand{\B} {\bmath{B}}

\renewcommand{\H} {\bmath{H}}
\newcommand{\I} {\bmath{I}}

\renewcommand{\S} {\bmath{S}}
\newcommand{\T} {\bmath{T}}
\renewcommand{\P} {\bmath{P}}

\newcommand{\R} {\bmath{R}}
\newcommand{\M} {\bmath{M}}

\newcommand{\U} {\bmath{U}}
\newcommand{\V} {\bmath{V}}

\renewcommand{\aa} {\bmath{a}}

\renewcommand{\u} {\bmath{u}}
\newcommand{\vv} {\bmath{v}}

\newcommand{\x} {\bmath{x}}
\newcommand{\y} {\bmath{y}}

\newcommand{\s} {\bmath{s}}
\newcommand{\ttt} {\bmath{t}}
\newcommand{\w} {\bmath{w}}
\newcommand{\z} {\bmath{z}}

\newenvironment{@abssec}[1]{%
	\if@twocolumn
	\section*{#1}%
	\else
	\vspace{.05in}\footnotesize
	\parindent .2in
	{\upshape\bfseries #1. }\ignorespaces 
	\fi}
{\if@twocolumn\else\par\vspace{.1in}\fi}

\newenvironment{keywords}{\begin{@abssec}{\keywordsname}}{\end{@abssec}}

\newenvironment{AMS}{\begin{@abssec}{\AMSname}}{\end{@abssec}}
\newcommand\keywordsname{Key words}
\newcommand\AMSname{AMS subject classifications}
\newcommand{\email}[1]{\protect\href{mailto:#1}{#1}}
\newcommand\funding[1]{\protect {\bfseries Funding:} #1}

\begin{document}
\title{Semi-Anchored Multi-Step Gradient Descent Ascent Method for Structured Nonconvex-Nonconcave Composite Minimax Problems\thanks{Submitted to the editors DATE.
		\funding{This work was supported in part by the National Research Foundation of Korea (NRF) grant funded by the Korea government (MSIT) (No. 2019R1A5A1028324), the POSCO Science Fellowship of POSCO TJ Park Foundation, and the Samsung Science \& Technology Foundation grant
			(No. SSTF-BA2101-02).}}}
\author{Sucheol Lee\thanks{Department of Mathematical Sciences, KAIST
		(\email{csfh1379@kaist.ac.kr}, \email{donghwankim@kaist.ac.kr}).}
	\and Donghwan Kim\footnotemark[2]}
\maketitle

\begin{abstract}
Minimax problems, such as 
generative adversarial network,
adversarial training,
and fair training,
are widely
solved by a multi-step gradient descent ascent (MGDA) method
in practice.
However, its convergence guarantee is limited.
In this paper,
inspired by the primal-dual hybrid gradient method,
we propose a new semi-anchoring (SA) technique for the MGDA method.
This makes the MGDA method 
find a stationary point of
a structured nonconvex-nonconcave composite minimax problem;
its saddle-subdifferential operator satisfies 
the weak Minty variational inequality condition.
The resulting method, named SA-MGDA, 
is built upon a Bregman proximal point method. 
We further develop its backtracking line-search version,
and its non-Euclidean version for smooth adaptable functions.
Numerical experiments, including a fair classification training, are provided.
\end{abstract}

\begin{keywords}
  minimax problem, nonconvex-nonconcave problem, multi-step gradient descent ascent method, 
  Bregman proximal point method
\end{keywords}

\begin{AMS}
  90C47, 90C26, 65K05, 47J25, 68T05
\end{AMS}

\section{Introduction}

Generative adversarial network~\cite{arjovsky:17:wga,goodfellow:14:gan},
adversarial training~\cite{kurakin:17:aml,madry:18:tdl}
and fair training~\cite{mohri:19:afl,nouiehed:19:sac}
involve solving a minimax problem:
\begin{equation}
\min_{\u\in\Us}\max_{\vv\in\Vs} \phi(\u,\vv) 
\label{eq:prob}
,\end{equation}
where $\Us\subseteq\reals^{d_u}$ and $\Vs\subseteq\reals^{d_v}$.
To solve~\eqref{eq:prob}, 
variants of the multi-step gradient descent ascent (MGDA) method,
which consists of one 
gradient descent update of $\u$
and multiple 
gradient ascent updates of $\vv$,
are widely used in practice
(see \eg, \cite{arjovsky:17:wga,brock:19:lsg,choi:18:sug,goodfellow:14:gan,gulrajani:17:ito,karras:19:asb,kurakin:17:aml,miyato:18:snf,madry:18:tdl}).
However, 
the MGDA method is unstable, \eg, for some standard cases
such as bilinear problems
and smooth convex-concave problems~\cite{goodfellow:16,mescheder:18:wtm}.
Therefore, this paper develops a new semi-anchoring (SA) technique 
that makes the MGDA find a stationary point of a structured nonconvex-nonconcave composite problem;
its saddle-subdifferential operator satisfies 
the weak Minty variational inequality (MVI) condition in~\cite{diakonikolas:21:emf}.
The weak MVI condition is weaker than the MVI condition
that has received recent attention
as one of standard nonconvex-nonconcave settings
in the optimization community~\cite{dang:15:otc,malitsky:20:gro}
and the machine learning community~\cite{mertikopoulos:19:omd,song:20:ode,zhou:17:smd}.

The proposed method, named SA-MGDA, is
built upon the Bregman proximal point (BPP) method
\cite{bauschke:03:bmo,borwein:11:aco,eckstein:93:npp} 
of a monotone operator,
such as the saddle-subdifferential operator of a convex-concave function.
Under a more general weak MVI condition~\cite{diakonikolas:21:emf},
we show that the worst-case rate of the BPP method 
(with a strongly convex and smooth Legendre function)
is $O(1/k)$, 
in terms of the Bregman distance
between the successive iterates,
where $k$ denotes the number of iterations.
We also show that the BPP method
has a linear rate
under the strong MVI condition in~\cite{song:20:ode,zhou:17:smd}.

The BPP method
is a nonlinear extension 
of a proximal point method~\cite{martinet:70:rdv,rockafellar:76:moa}
via a Bregman distance~\cite{bregman:67:trm}.
A specific choice of the Bregman distance in this paper
leads to the SA-MGDA method.
Therefore, such \emph{conceptual} SA-MGDA method,
requiring an exact maximization oracle on $\vv$,
consequently 
has the $O(1/k)$ worst-case rate, 
in terms of the associated Bregman distance
(and the squared subgradient norm),
for the structured nonconvex-nonconcave minimax problems.
Then, we show that its \emph{practical} version,
performing one (proximal) gradient descent update of $\u$
and 
a \emph{finite} number of 
``anchored'' (proximal) gradient ascent steps on $\vv$,
requires total $O(\epsilon^{-1}\log\epsilon^{-1})$ gradient steps
to find an $\epsilon$-stationary point.
This matches the $O(\epsilon^{-1})$ complexity of 
the \emph{conceptual} SA-MGDA up to a logarithmic factor.

The SA-MGDA is 
an implementation-wise simple modification of the MGDA 
with an improved convergence guarantee.
This is
inspired by and includes 
the primal-dual hybrid gradient (PDHG) method~\cite{chambolle:11:afo,esser:10:agf}.
The PDHG
is an instance of a preconditioned proximal point method~\cite{he:12:cao}
for a bilinearly-coupled minimax problem,
and our result is its nonlinear extension via the BPP.
Similar but different extensions of the PDHG have been recently studied in
\cite{hamedani:21:apd,yadav:18:san,zhao:19:osa}.

{\bf Our main contributions} are summarized as follows.

\begingroup
\allowdisplaybreaks
\begin{itemize}
	\item We study the properties of the BPP method 
	and analyze its worst-case rate
	(for a strongly convex and smooth Legendre function), 
	in terms of the Bregman distance between two successive iterates,
	under the weak MVI 
	and the strong MVI conditions, in Section~\ref{sec:bppm}.
	\item Built upon Section~\ref{sec:bppm},
	we develop a new semi-anchoring (SA) approach 
	for the MGDA, named SA-MGDA, 
	and provide its worst-case rates 
	for the structured nonconvex-nonconcave composite problems, in Section~\ref{sec:saag}.
	\item We construct a backtracking line-search version of the SA-MGDA,
	in Section~\ref{sec:exten},
	for the case where the Lipschitz constant is unavailable.
	We also develop a non-Euclidean version of the SA-MGDA
	for smooth adaptable functions
	\cite{bolte:18:fom}. 
\end{itemize}
\endgroup

\section{Related work} 

The MGDA can be viewed
as solving the following equivalent minimization problem
\cite{barazandeh:20:snc,jin:20:wil,nouiehed:19:sac}: 
\begin{align}
\min_{\u\in\Us}\Big\{\Psi(\u) := \max_{\vv\in\Vs}\phi(\u,\vv)\Big\}
\label{eq:prob,min}
\end{align}
by a gradient descent method on $\u$. 
This requires the following gradient computation:
\begin{align}
\nabla\Psi(\u) = \nabla_{\u}\phi(\u,\vv_*(\u))
\label{eq:danskin}
,\end{align} 
for $\vv_*(\u) := \argmax{\vv\in\Vs}\phi(\u,\vv)$, 
based on Danskin's theorem~\cite{danskin:67} 
(under additional conditions
for a differentiability of $\Psi$ explained soon).
This involves a maximization with respect to $\vv$ in~\eqref{eq:danskin},
\eg,
many number of gradient ascent updates of $\vv$ given $\u$. 
Unfortunately, even under the smoothness assumption on $\phi$, 
the function $\Psi$ is not differentiable in general.
Therefore, the analysis in~\cite{jin:20:wil}
relies on the theory of 
a subgradient method~\cite{davis:19:smb}
on $\u$,
resulting in a slow rate, 
under assumptions that $\phi$ is smooth and also Lipschitz continuous.

To make a function $\Psi$ differentiable, 
many existing literatures
either introduce 
assumptions on $\phi$
or apply a smoothing technique~\cite{nesterov:05:smo} 
to $\phi$ 
\cite{kong:21:aai,sanjabi:18:otc,zhao:20:apd}, 
or consider both
\cite{barazandeh:20:snc,lin:20:noa,nouiehed:19:sac,ostrovskii:21:eso,thekumparampil:19:eaf}.
These all further assume that $\Vs$ is a compact set. 
In specific, if a smooth function $\phi$ is further assumed to be strongly concave on $\vv$
with a convex compact set $\Vs$,
the function $\Psi$ is differentiable
and one can apply a gradient descent method on $\u$,
using~\eqref{eq:danskin}
\cite{barazandeh:20:snc,nouiehed:19:sac}.
In addition, in~\cite{nouiehed:19:sac}, 
the Polyak-\L{}ojasiewicz condition,
which is weaker than the strong concavity, 
on $\vv$
was assumed for a smooth function $\phi$ to make $\Psi$ smooth.

In many practical cases, 
desirable conditions such as a strong convexity are not given.
However, a smoothing technique~\cite{nesterov:05:smo}
can make 
the function $\Psi$ smooth,
while approximating the original problem.
This usually involves adding an appropriate regularization term on $\vv$,
such as $-\frac{\lambda}{2}||\vv - \tilde{\vv}||^2$, in~\eqref{eq:prob,min},
where $\lambda$ 
is a tuning parameter
and $\tilde{\vv}$ is some point.
For example, for a bilinearly-coupled problem of~\eqref{eq:prob,min}, 
\ie, $\u$ and $\vv$ are coupled only through a bilinear term $\u^\top\B\vv$,
a smoothing technique in~\cite{nesterov:05:smo} 
makes $\Psi$ smooth and enables us to use
a gradient descent method and its accelerated variants~\cite{nesterov:83:amf,nesterov:05:smo}.
Such regularization
has been 
further found to be useful practically, \eg,
in machine learning applications
\cite{gulrajani:17:ito,mescheder:18:wtm,roth:17:sto,sanjabi:18:otc}.
However, tuning a regularization parameter
and being unable to find
an exact solution of the original problem
remain problematic.

The proposed SA-MGDA method
does not require assumptions that 
$\phi$ is strongly concave on $\vv$ for a convergence guarantee.
It also does not 
need smoothing the function $\phi$ 
to make $\Psi$ smooth,
which involves a regularization parameter tuning.
The proposed semi-anchoring technique 
incorporates an \emph{anchor} point in the variable $\vv$ (hence named \emph{semi}-anchoring)
at each iteration,
which enforces the iterates to stay close to the anchor point
for stability.
This originates from the implicit regularization 
(smoothing) nature
of the Bregman proximal point method.
In particular, the anchoring 
chooses a convex combination of
the anchor point and an updated point as a next iterate.
This 
somewhat resembles Halpern iteration~\cite{halpern:67:fpo}
that uses an initial point as an anchor point throughout the entire process.
Inspired by Halpern iteration,
the paper~\cite{ryu:19:oao}
proposes a method called anchoring,
which is recently further studied in
\cite{lee:21:feg,yoon:21:aaf}. 
The anchoring-type techniques
also appear in
a conditional gradient method~\cite{frank:56:aaf} 
and a version of the celebrated Nesterov's fast gradient method
in~\cite{nesterov:05:smo}.

On the other hand,
single-step \emph{simultaneous} gradient descent ascent methods
are also gaining interest.
The standard version 
converges under a restrictive partial cocoercivity condition
for a Lipschitz continuous
saddle-differential operator
\cite{renaud:97:aeo}.
It also
works
when the function is strongly concave on $\vv$
\cite{chen:21:pgd}.
However, it diverges for a simple
bilinear case~\cite{mescheder:18:wtm,zhang:20:cog}
and for a more general problems.
Therefore, extragradient-type methods
\cite{hsieh:19:otc,korpelevich:76:aem,malitsky:20:gro,nesterov:07:dea,popov:80:amo}
with a better convergence behavior 
are gaining a great interest.
They are
also known under the name optimism
\cite{daskalakis:18:tgw,gidel:19:avi,mokhtari:20:cro,ryu:19:oao}.
Their nonasymptotic rates of convergence were initially studied in
\cite{nemirovski:04:pmw} for an averaged sequence,
and recently its last-iterate rate was studied in
\cite{golowich:20:lii,gorbunov:21:eml}.
The extragradient method
works under the MVI condition~\cite{dang:15:otc,mertikopoulos:19:omd},
and recently,
its variants, named EG+ \cite{diakonikolas:21:emf} and CEG+ \cite{pethick:22:elc}, have been found to work 
under the weak MVI condition.
Interestingly, the SA-MGDA method can be viewed as
an extragradient-type method
partially on $\vv$, similar to the PDHG.

Table~\ref{tab:comp} compares
the convergence guarantees of
the extragradient-type methods,
the MGDA and the proposed SA-MGDA,
under the problem settings of this paper's interest;
the monotonicity, the MVI, and the weak MVI.
Table~\ref{tab:comp} illustrates that
the SA-MGDA converges
under settings that the extragradient-type methods work,
unlike the MGDA\footnote{
\cite{barazandeh:20:snc,lin:20:noa,nouiehed:19:sac,thekumparampil:19:eaf}
analyze the complexity of the MGDA type methods for finding 
an $\epsilon$-stationary point of a (non)convex-concave problem
(including the monotonicity case,
but excluding the (weak) MVI case),
by adding a regularization term.
This, however, cannot find a stationary point exactly even with 
an exact maximization oracle,
unlike the SA-MGDA.
}, which is our main contribution.
Note that our contribution 
is not meant to compete 
with the extragradient type methods
in terms of how fast they find a stationary point 
with respect to the number of gradient computations.
Under the nonconvex-nonconcave setting, a preliminary work~\cite{jin:20:wil}
suggests that the MGDA type methods might 
find a stationary point
that can be preferred over the one found by the \emph{simultaneous} methods,
making the MGDA worth studying further.

\begingroup
\begin{table}[t]
    \renewcommand*{\arraystretch}{0.7}
    \centering
    \setlength{\tabcolsep}{2pt}
    \caption{Comparison of the problem settings for
        some representative single-step and multi-step methods 
        for structured nonconvex-nonconcave minimax problems of this paper's interest,
        under various assumptions on
        the Lipschitz continuous saddle-differential operator
        $\M_\phi = (\nabla_{\u}\phi,-\nabla_{\vv}\phi)$ 
        of a minimax problem~\eqref{eq:prob}.
        Some methods handle constrained problems~\eqref{eq:prob}
        or composite problems~\eqref{eq:P}.}
    \label{tab:comp}
    \begin{tabular}{cccccccc}
        \toprule[1pt]
        \multicolumn{3}{c}{Method} & \multirow{2}{*}{Monotone} 
        & MVI & Weak MVI & \multirow{2}{*}{Const.} & \multirow{2}{*}{Comp.} \\
        \cmidrule{1-3}
        Type & Name & Paper & & ($\rho=0$) & ($\rho>0$) & & \\
        \midrule[1pt]
        \multirow{5}{*}{Single-step} 
        & EG, DE & \cite{korpelevich:76:aem,nemirovski:04:pmw,nesterov:07:dea} & \cmark & & & \cmark \\
        & EG, OptDE & \cite{dang:15:otc,mertikopoulos:19:omd,song:20:ode} & \cmark & \cmark & & \cmark & \\
        & GRAAL & \cite{malitsky:20:gro} & \cmark & \cmark & & \cmark & \cmark \\
        \cmidrule[0.6pt]{2-8}
        & EG+ & \cite{diakonikolas:21:emf} & \cmark & \cmark & \cmark & & \\
        & CEG+ & \cite{pethick:22:elc} & \cmark & \cmark & \cmark & \cmark & \cmark \\
        \midrule[0.6pt]
        \multirow{2}{*}{Multi-step} & MGDA & \cite{barazandeh:20:snc,jin:20:wil,nouiehed:19:sac} & & & & \cmark & \cmark \\
        \cmidrule[0.6pt]{2-8}
        & SA-MGDA & This paper & \cmark & \cmark & \cmark & \cmark & \cmark \\
        \bottomrule[1pt]
    \end{tabular}
\end{table}
\endgroup

\section{Preliminaries}
\label{sec:pre}

\subsection{Bregman distance}

This paper uses a Legendre function~\cite{rockafellar:70}
and its associated Bregman distance~\cite{bregman:67:trm},
defined below,
as a non-Euclidean proximity measure
\cite{eckstein:93:npp,teboulle:18:asv}.
These help us to better handle the nonlinear geometry of a problem. 

\begin{definition}
	\cite{rockafellar:70}
	A function $h\;:\;\reals^d\to(-\infty,\infty]$ which is closed proper strictly convex and 
	essentially smooth\footnote{
	A closed convex proper function $h$ is called essentially smooth
	    if it satisfies the following three conditions:
	    (a) $\inter\dom h$ is nonempty;
	    (b) $h$ is differentiable throughout $\inter\dom h$;
	    (c) $\lim_{k\to\infty}||\nabla h(\x_k)|| = +\infty$
	    whenever $\x_1,\x_2,\ldots$, is a sequence in $\inter\dom h$
	        converging to a boundary point $\x$ of $\inter\dom h$
	        \cite{rockafellar:70}.} 
	will be called a Legendre function.
\end{definition}

\begin{definition}
	Let $h\;:\;\reals^d\to(-\infty,\infty]$ be a Legendre function.
	The Bregman distance associated to $h$, denoted by $D_h\;:\;\dom h\times\inter\dom h \to\reals_+$
	is defined by
	\begin{align*}
	D_h(\x,\y) := h(\x) - h(\y) - \inprod{\nabla h(\y)}{\x - \y}
	.\end{align*}
\end{definition}

A Bregman distance $D_h$ reduces to the Euclidean distance
for 
$h(\x) = \frac{1}{2}||\x||^2$.
$D_h$ is not symmetric in general, except for 
the case $h(\x) = \frac{1}{2}||\x||^2$.
In addition, $D_h(\x,\y)\ge0$ for all $(\x,\y)\in\dom h\times \inter\dom h$,
and $D_h(\x,\y)=0$ if and only if $\x=\y$ due to the strict convexity of $h$.
Popular examples of $h$ are
$h(\x) = \sum_{i=1}^d \frac{1}{p}|x_i|^p$ for $p\ge2$,
a Shannon entropy $h(\x) = \sum_{i=1}^d x_i\log x_i$, $\dom h=[0,\infty)^d$, 
and a Burg entropy $h(\x) = -\sum_{i=1}^d\log x_i$, $\dom h=(0,\infty)^d$.

An appropriate choice of $h$ from the above (partial) list,
especially that captures the geometry of the constraint sets of the problem,
has been found useful in many applications
(see, \eg,~\cite{bauschke:16:adl,beck:03:mda,bolte:18:fom,lu:18:rsc}).
This paper, however, chooses $h$ 
not from the above standard list.
Our choice of $h$ in Section~\ref{sec:saag}
is inspired by that of the PDHG
\cite{chambolle:11:afo,esser:10:agf,he:12:cao}.

\subsection{Composite minimax problem and weakly monotone operator}

We are interested in the minimax problem in a form: 
\begin{align}
\min_{\u\in\reals^{d_u}}\max_{\vv\in\reals^{d_v}} \left\{\Phi(\u,\vv) := f(\u) + \phi(\u,\vv) - g(\vv)\right\}
\label{eq:P}
,\end{align}
which satisfies the following assumption. Let $d:=d_u+d_v$.
\begin{assumption}
	\label{assum:smooth}
	$f\;:\;\reals^{d_u}\to(-\infty,\infty]$ 
	and $g\;:\;\reals^{d_v}\to(-\infty,\infty]$ 
	are closed, proper and convex functions.
	A function $\phi\;:\;\reals^d\to\reals$ is continuously differentiable and
	has a Lipschitz continuous gradient, \ie,
	there exists $L_{\u\u},L_{\u\vv},L_{\vv\u},L_{\vv\vv}>0$ such that,
	for all $\u,\bar{\u}\in\reals^{d_u}$ and $\vv,\bar{\vv}\in\reals^{d_v}$,
	\begin{align*}
	||\nabla_{\u}\phi(\u,\vv) - \nabla_{\u}\phi(\bar{\u},\bar{\vv})||
	&\le L_{\u\u}||\u-\bar{\u}|| + L_{\u\vv}||\vv-\bar{\vv}||, \\
	||\nabla_{\vv}\phi(\u,\vv) - \nabla_{\vv}\phi(\bar{\u},\bar{\vv})||
	&\le L_{\vv\u}||\u-\bar{\u}|| + L_{\vv\vv}||\vv-\bar{\vv}||.
	\end{align*}
\end{assumption}
Both the gradient $\nabla\phi$
and the saddle-differential of $\phi$, denoted by 
$\M_{\phi} := 
(\nabla_{\u}\phi,-\nabla_{\vv}\phi)$,
are $L$-Lipschitz continuous
(see Appendix~\ref{appx:lipschitz}), \ie, there exists $L>0$ such that
$||\M_\phi\x-\M_\phi\y||\le L||\x-\y||$
for all 
$\x,\y\in\reals^d$.

Finding a first-order stationary point $\x_*:=(\u_*,\vv_*)\in\reals^d$ of~\eqref{eq:P},
is equivalent to finding a zero of the following set-valued saddle-subdifferential operator
of $\Phi$: 
\begin{align}
\M := 
(\nabla_{\u}\phi + \partial f,
-\nabla_{\vv}\phi + \partial g)
\;:\;\reals^d\rightrightarrows(-\infty,\infty]^d
\label{eq:M}
.\end{align}
Let $X_*(\M) := \{\x_*\;:\;\zero\in\M\x_*\}$ be a nonempty solution set.
We consider
the squared subgradient norm at $\x$, 
$\min_{\s\in\M\x}||\s||^2$, 
as an optimality criteria, 
which is 
standard in 
nonconvex-nonconcave minimax problems \cite{diakonikolas:21:emf,lee:21:feg,pethick:22:elc}.

Under Assumption~\ref{assum:smooth}, the operator $\M$~\eqref{eq:M} satisfies
the following weakly monotone condition with 
$\gamma=\max\{L_{\u\u},L_{\vv\vv}\}$
(see Appendix~\ref{appx:weak}) and is maximal. 
\begin{assumption}
	\label{assum:M,weak}
	For some $\gamma\ge0$,
	an operator $\M$ is $\gamma$-weakly monotone, \ie,
	\begin{align*}
	\inprod{\x - \y}{\w - \z} \ge -\gamma||\x - \y||^2,
	\quad\forall (\x,\w),(\y,\z)\in\gra\M,
	\end{align*}
	where $\gra\M := \{(\x,\w)\in\reals^d\times\reals^d\;:\;\w\in\M\x\}$ denotes the graph of $\M$.
	Also, it is maximal, \ie,
	there exists no $\gamma$-weakly monotone operator
	that its graph
	properly contains $\gra\M$. 
\end{assumption}

\subsection{Structured nonconvex-nonconcave problem}

This section considers 
the structured nonconvex-nonconcave conditions; 
the weak Minty variational inequality (MVI) condition in~\cite{diakonikolas:21:emf}
and the strong MVI condition in~\cite{song:20:ode,zhou:17:smd}.

The MVI problem
is to find $\x_*$ such that 
\begin{align*}
\inprod{\x - \x_*}{\w} \ge 0,  \quad \forall(\x,\w)\in\gra\M.
\end{align*}
For a continuous $\M$, a solution set of the MVI problem is 
a subset of $X_*(\M)$,
and if $\M$ is monotone, they are equivalent.
The MVI condition, assuming that 
a solution of the MVI problem exists,
is studied in~\cite{dang:15:otc,malitsky:20:gro},
which is also studied under the name, the coherence, in
\cite{mertikopoulos:19:omd,song:20:ode,zhou:17:smd}.
Recently,~\cite{diakonikolas:21:emf}
introduced 
the following weaker condition, named weak MVI condition.

\begin{assumption}[Weak MVI]
	\label{assum:weakmvi}
	For some $\rho\ge 0$, there exists a solution $\x_*\in X_*(\M)$ such that
	\begin{align}
	\inprod{\x - \x_*}{\w} \ge -\frac{\rho}{2}||\w||^2, \quad \forall (\x,\w)\in\gra\M
	\nonumber
	.\end{align}
\end{assumption}
Let $X_*^\rho(\M)$ be the associated solution set.
Assumption~\ref{assum:weakmvi} 
is implied by the $-\frac{\rho}{2}$-comonotononicity~\cite{bauschke:21:gmo},
or equivalently the $\frac{\rho}{2}$-cohypomonotonicity~\cite{combettes:04:pmf}, \ie,
\begin{align*}
\inprod{\x - \y}{\w - \z} \ge -\frac{\rho}{2}||\w-\z||^2,
\quad\forall (\x,\w),(\y,\z)\in\gra\M.
\end{align*}
The comonotonicity is also implied by the $\alpha\ge0$-interaction dominant condition 
in~\cite{grimmer:20:tlo} 
(see \cite[Example 1]{lee:21:feg}).
\cite[Proposition 2]{daskalakis:20:ipg}\cite{diakonikolas:21:emf} 
consider a (constrained)
two-agent zero-sum reinforcement learning problem,
called the von Neumann ratio game, 
that satisfies
the weak MVI condition, but neither 
the MVI condition nor 
the comonotonicity condition. 
\cite{pethick:22:elc}
also provides examples satisfying the weak MVI condition.

We also consider the strong MVI condition in~\cite{song:20:ode,zhou:17:smd}.
This condition
is also non-monotone, 
and 
includes
the $\mu$-strong pseudomonotonicity~\cite{nguyen:20:sro}
(see~\cite{nguyen:20:sro} for examples).
Let $S_*^{\mu}(\M)$ be the associated solution set.
\begin{assumption}[Strong MVI]
	\label{assum:strongmvi}
	For some $\mu\ge 0$, there exists a solution $\x_*\in X_*(\M)$ such that
	\begin{align*}
	\inprod{\x - \x_*}{\w} \ge \mu||\x - \x_*||^2 , \quad \forall (\x,\w)\in\gra\M.
	\end{align*}
\end{assumption}

\section{Bregman proximal point (BPP) method}
\label{sec:bppm}

\subsection{The $h$-resolvent}

The $h$-resolvent of a 
monotone operator $\M$ with respect to a Legendre function $h$ is defined as~\cite{eckstein:93:npp}
\begin{align}
\R_{\M}^h := (\nabla h + \M)^{-1}\nabla h
\nonumber
,\end{align}
where we omit $\M$ and $h$ in $\R_{\M}^h$
for simplicity hereafter, unless necessary.
This reduces to the standard resolvent operator $(\I + \M)^{-1}$ for $h = \frac{1}{2}||\cdot||^2$,
where $\I$ is an identity operator.
The $h$-resolvent $\R$
is single-valued on its domain
for a monotone operator $\M$
\cite[Proposition 3.8]{bauschke:03:bmo},
and we extend this for a weakly monotone operator.

\begin{lemma}
	\label{lem:single}
	Let $\M$ satisfy Assumption~\ref{assum:M,weak} for some $\gamma\ge0$,
	and $h$ be a $\mu_h$-strongly convex Legendre function.
	Then, if $\mu_h > \gamma$, the $h$-resolvent $\R$ is single-valued on $\inter\dom h$. 
\end{lemma}
\begin{proof}
Note that $\nabla h + \M = (\nabla h -\gamma\I) + (\M+\gamma\I)$. From the condition that $\mu_h > \gamma$ and $\M$ is $\gamma$-weakly monotone, it is straightforward to show that $h-\frac{\gamma}{2}\|\cdot\|^2$ is a Legendre function 
and $\M+\gamma\I$ is maximally monotone. Then 
\cite[Corollary 2.3]{bauschke:10:grf} shows that $\ran(\nabla h + \M) = \reals^d$. 
This implies that $\R \x$ is nonempty for all $\x\in\inter\dom h$. Assume that $\y,\z \in \R\x$.
Since 
$\nabla h(\x)-\nabla h(\y) \in \M\y$
and $\nabla h(\x) - \nabla h(\z) \in \M\z$, we have the following inequality:
\begin{align*}
-\gamma\|\y-\z\|^2\le -\inprod{\nabla h(\y)-\nabla h(\z)}{\y-\z} \le - \mu_h \|\y-\z\|^2.
\end{align*}
So if $\mu_h > \gamma$,
the inequality implies that $\y = \z$.
\end{proof}

\subsection{BPP under weak MVI condition}

The BPP method~\cite{eckstein:93:npp} iteratively applies 
the $h$-resolvent as, for $k=0,1,\ldots$,
\begin{align}
\x_{k+1} = \R(\x_k)
\label{eq:bppm}
,\end{align}
which converges to a zero of 
a monotone operator $\M$. 
We analyze the worst-case convergence behavior of the BPP method
under Assumptions~\ref{assum:M,weak} and~\ref{assum:weakmvi}  (or~\ref{assum:strongmvi}).
We first state the Bregman nonexpansive property of $\R$ 
below.
For $\rho=0$ and for any Legendre function $h$,
this reduces to the quasi-Bregman firmly nonexpansive property
$D_h(\x_*,\R\x) \le D_h(\x_*,\x) - D_h(\R\x,\x)$
\cite{borwein:11:aco,eckstein:93:npp}.

\begin{lemma}
	\label{lem:expansive,weakmvi}
	Let $\M$ satisfy Assumption~\ref{assum:weakmvi} for some $\rho\ge0$,
	and $h$ be an $L_h$-smooth Legendre function.
	Then, 
	if $\R\x$ exists, for any $\x_*\in X_*^\rho(\M)$,
	\begin{align*}
	D_h(\x_*,\R\x) \le D_h(\x_*,\x) - (1 - \rho L_h)D_h(\R\x,\x)
	.\end{align*} 
\end{lemma}
\begin{proof}
By the definition of $\R\x$, we have $\nabla h(\x) - \nabla h(\R\x) \in \M\R\x$.
Then, Assumption~\ref{assum:weakmvi} on $\M$ implies that
\begin{align*}
0 &\le \inprod{\nabla h(\x) - \nabla h(\R\x)}{\R\x - \x_*} + \frac{\rho}{2}||\nabla h(\x) - \nabla h(\R\x)||^2 \\
&= \inprod{\nabla h(\x)}{\R\x - \x} - \inprod{\nabla h(\x)}{\x_* - \x} 
+ \inprod{\nabla h(\R\x)}{\x_* - \R\x} \\
&\quad + \frac{\rho}{2}||\nabla h(\x) - \nabla h(\R\x)||^2 \\
&= -D_h(\R\x,\x) + D_h(\x_*,\x) - D_h(\x_*,\R\x) + \frac{\rho}{2}||\nabla h(\x) - \nabla h(\R\x)||^2 \\
&\le -D_h(\R\x,\x) + D_h(\x_*,\x) - D_h(\x_*,\R\x) + \rho L_hD_h(\R\x,\x)
,\end{align*}
where the last inequality follows from the convex and $L_h$-smooth properties of $h$,
\ie, $\frac{1}{2L_h}||\nabla h(\x) - \nabla h(\y)||^2 \le D_h(\x,\y)$
for all $\x,\y$
\cite[Theroem 2.1.5]{nesterov:18}.
\end{proof}
This 
presents that the condition $\rho L_h\le 1$ 
guarantees
the quasi-Bregman nonexpansiveness $D_h(\x_*,\R\x) \le D_h(\x_*,\x)$.
%
We then have the following worst-case rate in terms of the best Bregman distance between two successive iterates
under the weak MVI condition,
and the convergence property of the iterate sequences.
These built upon \cite[Theorem 1]{eckstein:93:npp}
of the BPP for a monotone operator.
\begin{theorem}
	\label{thm:bppm,weakmvi}
	Let $\M$ satisfy Assumptions~\ref{assum:M,weak} and~\ref{assum:weakmvi} 
	for some $\gamma,\rho\ge0$, 
	and $h$ be a $\mu_h$-strongly convex and $L_h$-smooth Legendre function 
	with $\mu_h > \gamma$ and $\rho L_h<1$, respectively.
	Then, the sequence $\{\x_k\}$ of
	the BPP method~\eqref{eq:bppm} satisfies, for $k\ge1$ and for any $\x_*\in X_*^\rho(\M)$,
	\begin{align}
	\min_{i=1,\ldots,k} D_h(\x_i,\x_{i-1}) \le \frac{D_h(\x_*,\x_0)}{\left(1 - \rho L_h\right)k}
	\nonumber
	.\end{align}
	Moreover, 
	all limit points of the sequence $\{\x_k\}$ of the BPP method
	are in $X_*(\M)$,
	and if we further assume that
	$X_*^\rho(\M) = X_*(\M)$, 
	the sequence $\{\x_k\}$ 
	converges to a solution $\x_*\in X_*^\rho(\M)$.
\end{theorem}
\begin{proof}
By Lemma~\ref{lem:single}, the condition $\mu_h>\gamma$ implies that 
$\R\x$ exists for any $\x$.
By Lemma~\ref{lem:expansive,weakmvi}, we get
\begin{align}
0 \le D_h(\x_*,\x_i) - D_h(\x_*,\x_{i+1}) - (1-\rho L_h)D_h(\x_{i+1},\x_i) 
\label{eq:bppm,weakmvi}
\end{align}
for all $i\ge 0$.
By summing over the above inequality,  
we get
\begin{align*}
\sum_{i=1}^k (1-\rho L_h)D_h(\x_i,\x_{i-1}) &\le \sum_{i=1}^k \left(D_h(\x_*,\x_{i-1}) - D_h(\x_*,\x_i) \right)\\
&= D_h(\x_*,\x_0) - D_h(\x_*,\x_k)\\
&\le D_h(\x_*,\x_0).
\end{align*}
Hence, by dividing the both sides of the inequality by $(1-\rho L_h)k$, we get
\begin{align*}
\min_{i=1,\ldots,k} D_h(\x_i,\x_{i-1})
\le \frac{1}{k}\sum_{i=1}^k D_h(\x_i,\x_{i-1})
\le \frac{D_h(\x_*,\x_0)}{(1-\rho L_h)k}
.\end{align*}

Next, we prove 
that all limit point of the sequence $\{\x_k\}$
are first-order stationary points.
By \eqref{eq:bppm,weakmvi}, the sequence $\{D_h(\x_*,\x_k)\}$ is bounded above, which implies that
$\{\x_k\}$ is a bounded sequence
due to the strong convexity of $h$. 
Let $\x_\infty$ be any limit point of $\{\x_k\}$,
and let
$\{\x_{k(j)}\}$ be a subsequence 
that converges to $\x_\infty$.
Note that $D_h(\x_{k+1},\x_k)\to0$ 
(and $\nabla h (\x_{k}) - \nabla h (\x_{k+1}) \to \zero$ 
by the convex and $L_h$-smooth properties of $h$,
	\ie, $\frac{1}{2L_h}||\nabla h(\x) - \nabla h(\y)||^2 \le D_h(\x,\y)$
\cite[Theroem 2.1.5]{nesterov:18})
as $k\to\infty$,
since $\sum_{i=1}^kD_h(\x_i,\x_{i-1}) \le \frac{D_h(\x_*,\x_0)}{1-\rho L_h}$.
Then $\{\x_{k(j)+1}\}$ also converges to $\x_\infty$ since
\begin{align*}
    ||\x_{k(j)+1} - \x_\infty||^2 &\le 2||\x_{k(j)}-\x_\infty||^2 + 2||\x_{k(j)+1} - \x_{k(j)}||^2\\
    &\le 2||\x_{k(j)}-\x_\infty||^2 + \frac{4}{\mu_h}D_h(\x_{k(j)+1},\x_{k(j)})
,\end{align*}
where the last inequality uses the strong convexity of $h$,
\ie, $\frac{\mu_h}{2}||\x_{k(j)+1} - \x_{k(j)}||^2 \le D_h(\x_{k(j)+1},\x_{k(j)})$.
Since $\M + \gamma\I$ is maximally monotone
and satisfies  $\nabla h (\x_{k(j)}) - \nabla h (\x_{k(j)+1}) + \gamma \x_{k(j)+1}
\in (\M + \gamma\I)(\x_{k(j)+1})$,
we finally have $\zero \in \M\x_\infty$
by \cite[Proposition 20.32]{bauschke:11:caa}.

Lastly, assume that $X_*^\rho(\M) = X_*(\M)$.
Then $\x_\infty$ is in $X_*^\rho(\M)$, and
since $\lim_{j\rightarrow \infty} D_h(\x_\infty,\x_{k(j)})=0$ and $\{D_h(\x_\infty,\x_k)\}$ is a nonincreasing sequence by 
\eqref{eq:bppm,weakmvi},
we get $\lim_{k\rightarrow \infty} D_h(\x_\infty,\x_k)=0$. 
Therefore, by the strong convexity of $h$, $\{\x_k\}$ converges to $\x_\infty \in X_*^\rho(\M)$.
\end{proof}

The BPP has a linear rate
under the strong MVI condition.
\begin{theorem}
	\label{thm:bppm,str,conv}
	Let $\M$ 
	satisfy Assumptions~\ref{assum:M,weak} and~\ref{assum:strongmvi} 
	for some $\gamma, \mu\ge0$,
	and $h$ be a $\mu_h$-strongly convex and $L_h$-smooth Legendre function 
	with $\mu_h > \gamma$.
	Then, for $k\ge1$ and for any $\x_*\in S_*^{\mu}(\M)$, 
	the sequence $\{\x_k\}$ of the BPP method~\eqref{eq:bppm}
	satisfies
	\begin{align*}
	D_h(\x_*,\x_k) \le 
		\left(\frac{2\mu}{L_h}+1\right)^{-k}
        D_h(\x_*,\x_{0})
	.\end{align*}
\end{theorem}
\begin{proof}
By Lemma~\ref{lem:single}, the condition $\mu_h>\gamma$ implies that 
$\R\x$ exists for any $\x$.
By the definition of $\R\x$, we have $\nabla h(\x) - \nabla h(\R\x) \in \M\R\x$.
Then, Assumption~\ref{assum:strongmvi} on
$\M$ implies that
\begin{align*}
\mu||\x_* - \R\x||^2 &\le \inprod{\nabla h(\x) - \nabla h(\R\x)}{\R\x - \x_*} \\
&= -D_h(\R\x,\x) + D_h(\x_*,\x) - D_h(\x_*,\R\x)
.\end{align*}
By letting $\x = \x_{i-1}$
and using the $L_h$-smoothness of $h$, 
\ie, $D_h(\x_*,\x_i) \le \frac{L_h}{2}||\x_* - \x_i||^2$,
we have
\begin{align*}
\left(\frac{2\mu}{L_h}+1\right)D_h(\x_*,\x_i) \le -D_h(\x_i,\x_{i-1}) + D_h(\x_*,\x_{i-1})
\le D_h(\x_*,\x_{i-1})
.\end{align*}
\end{proof}

\subsection{Example: primal-dual hybrid gradient method}
\label{sec:pdhg}

This section considers the following bilinearly-coupled convex-concave minimax problems
\begin{align}
\min_{\u\in\reals^{d_u}}\max_{\vv\in\reals^{d_v}} 
\left\{
f(\u) + \inprod{\u}{\B\vv} - g(\vv)\right\}
\label{eq:P,bilinear}
,\end{align}
which is an instance of~\eqref{eq:P}.
%
The following choice of $h$ with $\tau\in\big(0,\frac{1}{||\B||}\big)$:
\begin{align}
h(\u,\vv) = \frac{1}{2\tau}\left(||\u||^2 + ||\vv||^2\right) - \inprod{\u}{\B\vv}
\label{eq:h,pdhg}
\end{align}
yields a (linear) preconditioner~\cite{he:12:cao}
\begin{align*}
\nabla h(\u,\vv) = \P(\u,\vv) := 
\left[\begin{array}{cc}
\frac{1}{\tau}\I & -\B \\
-\B^\top & \frac{1}{\tau}\I
\end{array}\right]
\left[\begin{array}{c}
\u \\ \vv
\end{array}\right]
.\end{align*}
In~\cite{he:12:cao},
the resulting Bregman (preconditioned) proximal point method
is shown to be equivalent to the PDHG~\cite{chambolle:11:afo,esser:10:agf}:
\comment{
\begin{align*}
    \left[\begin{array}{c}
\u_{k+1} \\ \vv_{k+1}
\end{array}\right] &= (\P + \M)^{-1}\P \left[\begin{array}{c}
\u_k \\ \vv_k
\end{array}\right]\\
    &= \left(\left[\begin{array}{cc}
\frac{1}{\tau}\I & -\B \\
-\B^\top & \frac{1}{\tau}\I
\end{array}\right] + 
\left[\begin{array}{cc}
\partial f & \B \\
-\B^\top & \partial g
\end{array}\right]\right)^{-1}
\left[\begin{array}{cc}
\frac{1}{\tau}\I & -\B \\
-\B^\top & \frac{1}{\tau}\I
\end{array}\right]
\left[\begin{array}{c}
\u_k \\ \vv_k
\end{array}\right]\\
&= \left[\begin{array}{cc}
\I + \tau\partial f & \zero \\
-2\tau\B^\top & \I + \tau\partial g
\end{array}\right]^{-1}
\left[\begin{array}{cc}
\I & -\tau\B \\
-\tau\B^\top & \I
\end{array}\right]
\left[\begin{array}{c}
\u_k \\ \vv_k
\end{array}\right],
\end{align*}
for $k=0,1,\ldots$, which can be rewritten as the followings.
}
\begin{align*}
    \u_{k+1} &= \prox_{\tau f}\left[\u_k - \tau\B\vv_k\right], \\
    \vv_{k+1} &= \prox_{\tau g}\left[\vv_k + \tau\B^\top(2\u_{k+1} - \u_k)\right],
\end{align*}
for $k=0,1,\ldots$,
where the proximal operator is defined as
$\prox_{\psi} := (\I + \partial\psi)^{-1}$.



The PDHG is similar to the alternating (proximal) gradient 
descent ascent method,
except that the PDHG uses the term
$ 
\B^\top(2\u_{k+1} - \u_k)
$ 
in the $\vv_{k+1}$ update,
instead of $\B^\top\u_{k+1}$. 
This simple modification improves convergence,
which is a one-sided variant
of the extragradient-type methods
\cite{hsieh:19:otc,korpelevich:76:aem,malitsky:20:gro,popov:80:amo}.


By Theorem~\ref{thm:bppm,weakmvi}
and a stronger nonexpansive property 
of $\R$ associated with the PDHG, 
\ie, 
$\inprod{\R\x - \R\y}{\P(\R\x - \R\y)} \le 
\inprod{\x - \y}{\P(\x - \y)}$ for all $\x,\y\in\reals^d$,
the last iterate of the PDHG satisfies
$ 
D_h(\x_k,\x_{k-1}) = \frac{1}{2}
\inprod{\x_k - \x_{k-1}}{\P(\x_k - \x_{k-1})}
\le O(1/k) 
$,
where $\x_k := (\u_k,\vv_k)$. 
In~\cite{kim:21:app},
the rate is improved to a fast $O(1/k^2)$ rate 
by a new momentum technique.
A more widely-known rate of PDHG for the averaged iterate $\bar{\x}_k = \frac{1}{k}\sum_{i=1}^k\x_i$,
in terms of the primal-dual gap, is studied
in~\cite{chambolle:11:afo}.

When $\B$ is a square matrix with full rank, 
the PDHG has a linear rate
$1-\tau^2\sigma_{\min}^2(\B)$,
similar to
its ODE analysis in~\cite{yadav:18:san}.
\begin{proposition}
	\label{prop:linear}
	Assume that $f=g=0$, and $\B$ is a square matrix with full rank.
	Let $\{\x_k\}$ be generated by the PDHG.
	Then, for any $\epsilon > 0$, there exists $K\ge0$ such that,
	for all $k\ge K$,
	\begin{align}
	||\hat{\x}_k - \hat{\x}_*||^2 \le 
	\Big(\sqrt{1 - \tau^2\sigma_{\min}^2(\B)} + \epsilon\Big)^{2k}
	||\hat{\x}_1 - \hat{\x}_*||^2
	\nonumber
	\end{align}
	where $\hat{\x}_k := (\x_k,\x_{k-1})$
	and
	$\hat{\x}_* := (\x_*,\x_*) = \zero$. 
\end{proposition}
\begin{proof}
See Appendix~\ref{appx:linear}.
\end{proof}
The standard proximal point method 
(the BPP method with 
$h(\x) = \frac{1}{2}||\x||^2$)
has a linear rate
$\frac{1}{1+\tau^2\sigma_{\min}^2(\B)}$ for bilinear problems~\cite{rockafellar:70:moa},
which is slower than
that
of the PDHG.
In~\cite{mokhtari:20:aua}, the extragradient type method has a rate 
similar but slower than that of the proximal point method,
implying that there is potential in further studying the PDHG type method. 

The PDHG has been extended to tackle nonlinear problems~\eqref{eq:P}
in~\cite{hamedani:21:apd,yadav:18:san,zhao:19:osa}
by generalizing the term
$\B^\top(2\u_{k+1} - \u_k)$.
In particular, under the setting different from this paper,
the prediction method in~\cite{yadav:18:san} uses the term $\nabla_{\vv}\phi(2\u_{k+1} - \u_k,\vv_k)$,
and papers~\cite{hamedani:21:apd,zhao:19:osa} consider
$2\nabla_{\vv}\phi(\u_k,\vv_k) - \nabla_{\vv}\phi(\u_{k-1},\vv_{k-1})$
with different update ordering.
Next section proposes a different extension of the PDHG method
by choosing a specific nonlinear function $h$ for the BPP method,
leading to the 
\emph{semi}-implicit gradient term 
$
2\nabla_{\vv}\phi(\u_{k+1},\vv_{k+1}) - \nabla_{\vv}\phi(\u_k,\vv_k)
$.

\section{Semi-anchored multi-step gradient descent ascent (SA-MGDA)} 
\label{sec:saag}

\subsection{
Constructing SA-MGDA from BPP}

Inspired by $h$~\eqref{eq:h,pdhg} of the PDHG,
this section considers the following:
\begin{align}
h(\u,\vv) = \frac{1}{2\tau}\left(||\u||^2 + ||\vv||^2\right) - \phi(\u,\vv)
\label{eq:h,saag}
\end{align}
under Assumption~\ref{assum:smooth}, 
which
is $\left(\frac{1}{\tau}-L\right)$-strongly convex\footnote{ 
We have that
$\inprod{\nabla h(\x) - \nabla h(\y)}{\x - \y} 
= \Inprod{\Big(\frac{1}{\tau}\I - \nabla\phi\Big)\x - \Big(\frac{1}{\tau}\I - \nabla\phi\Big)\y}{\x - \y} 
= \frac{1}{\tau}||\x-\y||^2  - \inprod{\nabla\phi(\x) - \nabla\phi(\y)}{\x-\y}
\ge \frac{1}{\tau}||\x-\y||^2 - ||\nabla\phi(\x) - \nabla\phi(\y)||\;||\x-\y|| 
\ge \left(\frac{1}{\tau} - L\right)||\x-\y||^2$
for all $\x,\y\in\reals^d$,
where the last inequality uses the $L$-Lipschitz continuity of $\nabla\phi$.}
when $\frac{1}{\tau} > L$.
This
yields the Bregman distance
\begin{align*}
D_h(\x,\y)
= \phi(\y) + \inprod{\nabla\phi(\y)}{\x - \y} + \frac{1}{2\tau}||\x - \y||^2 
- \phi(\x)
,\end{align*} 
which is a difference 
between the function $\phi$
and its quadratic upper bound at $\y$.
This is a non-standard optimality measure in minimax optimization,
but due to the smoothness and convexity of $h$,
we can show that this upper bounds
the squared subgradient norm
that is a standard optimality measure 
in nonconvex-nonconcave minimax optimization
\cite{diakonikolas:21:emf,lee:21:feg,pethick:22:elc}.

Since $\M$ in \eqref{eq:M} is $\hat{L}:=\max\{L_{\u\u},\,L_{\vv\vv}\}$-weakly monotone, 
the corresponding BPP update~\eqref{eq:bppm}
with $\nabla h = \frac{1}{\tau}\I - \nabla\phi$ is well-defined for $\frac{1}{\tau}-L>\hat{L}$,
by Lemma~\ref{lem:single}.
%
%
The BPP update~\eqref{eq:bppm} with $h$ in~\eqref{eq:h,saag} 
is
\begin{align*}
\left[\begin{array}{c}
\u_{k+1} \\ \vv_{k+1}
\end{array}\right]
= \left(\frac{1}{\tau}\I - \nabla\phi + \M\right)^{-1}
\left(\frac{1}{\tau}\I - \nabla\phi\right)
\left[\begin{array}{c}
\u_k \\ \vv_k
\end{array}\right]
.\end{align*}
This is equivalent to the followings:
\begin{align*}
&\frac{1}{\tau}\u_{k+1} - \nabla_{\u}\phi(\u_{k+1},\vv_{k+1}) + \nabla_{\u}\phi(\u_{k+1},\vv_{k+1})
+ \partial f(\u_{k+1})
\;\ni\;
\frac{1}{\tau}\u_k - \nabla_{\u}\phi(\u_k,\vv_k), \\
&\frac{1}{\tau}\vv_{k+1} - \nabla_{\vv}\phi(\u_{k+1},\vv_{k+1}) - \nabla_{\vv}\phi(\u_{k+1},\vv_{k+1})
+ \partial g(\vv_{k+1})
\;\ni\;
\frac{1}{\tau}\vv_k - \nabla_{\vv}\phi(\u_k,\vv_k), \nonumber
\end{align*}
and rewriting the above in the minimization and maximization form 
respectively yields 
the followings.
%
\begingroup
\allowdisplaybreaks
\begin{align}
\u_{k+1} &= \argmin{\u\in\reals^{d_u}}\Big\{\phi(\u_k,\vv_k) + 
\inprod{\nabla_{\u}\phi(\u_k,\vv_k)}{\u - \u_k} + \frac{1}{2\tau}||\u - \u_k||^2 + f(\u)\Big\}
\label{eq:uv} \\
\vv_{k+1} &= \argmax{\vv\in\reals^{d_v}} 
\Big\{
2\phi(\u_{k+1},\vv) \nonumber \\
&\hspace{55pt} - \phi(\u_k,\vv_k)
- \inprod{\nabla_{\vv}\phi(\u_k,\vv_k)}{\vv- \vv_k} - \frac{1}{2\tau}||\vv - \vv_k||^2 - g(\vv)\Big\} \nonumber \\
&= \argmax{\vv\in\reals^{d_v}}  \Big\{
2\phi(\u_{k+1},\vv) - \frac{1}{2\tau}||\vv - (\vv_k - \tau\nabla_{\vv}\phi(\u_k,\vv_k))||^2
 - g(\vv)\Big\}.  
\nonumber
\end{align}
\endgroup
The minimization in $\u$ can be solved by one proximal gradient update,
while the maximization in $\vv$
is equivalent to
an implicit update 
\begin{align*}
\vv_{k+1} = \prox_{\tau g}[\vv_k 
+ \tau(2\nabla_{\vv} \phi(\u_{k+1},\vv_{k+1}) - \nabla_{\vv}\phi(\u_k,\vv_k))]
.\end{align*}
In other words,
an iterative method is required for the maximization 
in $\vv$.
This maximization problem
consists of a $\big(\frac{1}{\tau} + 2L_{\vv\vv}\big)$-smooth 
and $\big(\frac{1}{\tau}-2L_{\vv\vv}\big)$-strongly concave function,
and a concave but possibly nonsmooth function $-g$.
We thus use 
a proximal gradient method
for total $J$ number of (inner) iterations 
for the (inexact) maximization in $\vv$.
The corresponding proximal gradient ascent steps 
are, for $j=0,1,\ldots,J-1$,
\begin{align}
\vv_{k,j+1} = \argmax{\vv\in\reals^{d_v}}  \Big\{&
2\phi(\u_{k+1},\vv_{k,j}) + \inprod{2\nabla_{\vv}\phi(\u_{k+1},\vv_{k,j})}{\vv - \vv_{k,j}}
 - L_{\vv\vv}||\vv - \vv_{k,j}||^2 \nonumber\\
 &- \frac{1}{2\tau}||\vv - (\vv_k - \tau\nabla_{\vv}\phi(\u_k,\vv_k))||^2
 - g(\vv)\Big\}
\label{eq:vvkjp} 
.\end{align}
Note that
fast proximal gradient methods~\cite{beck:09:afi,chambolle:16:ait,nesterov:83:amf}
can be used for acceleration.
The resulting SA-MGDA method is 
illustrated in Algorithm~\ref{alg:saag}
(for $L_{\vv\vv} > 0$).
The explicit update of $\vv_{k,j+1}$ in Algorithm~\ref{alg:saag} can be derived by
examining the optimality condition of~\eqref{eq:vvkjp}:
\begin{align*}
2\nabla_{\vv}\phi(\u_{k+1},\vv_{k,j}) - 2L_{\vv\vv}(\vv - \vv_{k,j})
-\frac{1}{\tau}(\vv - (\vv_k - \tau\nabla_{\vv}\phi(\u_k,\vv_k)))
- \partial g(\vv) 
\;\ni\;
\zero
.\end{align*}
%
%
The explicit update of $\vv_{k,j+1}$ 
in Algorithm~\ref{alg:saag}
involves a convex combination of an \emph{anchor} point%
, $\vv_k - \tau\nabla_{\vv}\phi(\u_k,\vv_k)$,
that only depends on the previous point $(\u_k,\vv_k)$
and a recent point%
, $\vv_{k,j}+ \frac{1}{L_{\vv\vv}}\nabla_{\vv}\phi(\u_{k+1},\vv_{k,j})$,
that depends on $(\u_{k+1},\vv_{k,j})$.
Since the anchoring appears only in the variable $\vv$,
we name this approach as \emph{semi-anchoring}.
%

\begin{algorithm}[h]
	\caption{SA-MGDA for~\eqref{eq:P}
	   with $\tau\in\big(\frac{\rho}{1-\rho L},\frac{1}{L+\hat{L}}\big)$, $\eta = \frac{\tau}{1 + 2L_{\vv\vv}\tau}$} 
	\label{alg:saag}
	\begin{algorithmic}
		\For{$k=0,1,\ldots$}
		\State{$\u_{k+1} = \prox_{\tau f}\left[\u_k - \tau\nabla_{\u}\phi(\u_k,\vv_k)\right]$,
		\; $\vv_{k,0} = \vv_k$
		}
		\For{$j=0,\ldots,J-1$}
		\State{$\vv_{k,j+1}
		= \prox_{\eta g}\Big[\frac{\eta}{\tau}\left(\vv_k 
		- \tau\nabla_{\vv}\phi(\u_k,\vv_k)\right)+ 2\eta L_{\vv\vv}\big(\vv_{k,j}
		+ \frac{1}{L_{\vv\vv}}\nabla_{\vv}\phi(\u_{k+1},\vv_{k,j})\big)\Big]$
		}
		\EndFor
		\State{$\vv_{k+1} = \vv_{k,J}$,
		\; $\x_{k+1} = (\u_{k+1},\vv_{k+1})$
		}
		\EndFor
	\end{algorithmic}
\end{algorithm}

\subsection{SA-MGDA with $J=\infty$}

The following theorems of the SA-MGDA method
with $J=\infty$ (or equivalently, with an exact maximization oracle)
are byproducts of Lemma~\ref{lem:single} 
and Theorems~\ref{thm:bppm,weakmvi} and~\ref{thm:bppm,str,conv}
of the BPP method, 
for a specific $h$ in~\eqref{eq:h,saag}
that is $\mu_h$-strongly convex and $L_h$-smooth
with $\mu_h=\frac{1}{\tau}-L$ and $L_h = \frac{1}{\tau}+L$. 

\begin{theorem}
	\label{thm:saag,weakmvi}
	Let $\M$~\eqref{eq:M} of the composite problem~\eqref{eq:P} 
	satisfy Assumption~\ref{assum:weakmvi}
	for some $\rho\in\big[0,\frac{1}{2L+\hat{L}}\big)$, 
	and let $f,g$ and $\phi$ satisfy Assumption~\ref{assum:smooth}.
	Then, the sequence $\{\x_k\}$
	of the SA-MGDA (with $J=\infty$) 
	satisfies, for $k\ge1$, $\tau\in\big(\frac{\rho}{1-\rho L},\frac{1}{L+\hat{L}}\big)$
	and for any $\x_*\in X_*^\rho(\M)$,
	\begin{align*}
	\min_{i=1,\ldots,k} \min_{\s_i\in\M\x_i} \frac{||\s_i||^2}{2\left(\frac{1}{\tau}+L\right)}
	\le
	\min_{i=1,\ldots,k} D_h(\x_i,\x_{i-1})
	\le&\; \frac{D_h(\x_*,\x_0)}{\left(1 - \rho\left(\frac{1}{\tau}+L\right)\right)k}
	.\end{align*}
	Moreover, 
	all limit points of the sequence $\{\x_k\}$ of the SA-MGDA method
	are in $X_*(\M)$,
	and if we further assume that
	$X_*^\rho(\M) = X_*(\M)$,
	the sequence $\{\x_k\}$ 
	converges to a solution $\x_*\in X_*^\rho(\M)$.
\end{theorem}
\begin{proof}
The proof follows from Theorem~\ref{thm:bppm,weakmvi}
	with constraints 
	$\mu_h>\hat{L}$ 
	and 
	$\rho L_h < 1$,
	yielding $\tau < \frac{1}{L+\hat{L}}$ and $\tau > \frac{\rho}{1-\rho L}$.
	We also need $\rho < \frac{1}{2L + \hat{L}}$,
	so that $\tau$ exists.
	The first inequality follows from 
	$\nabla h(\x_i) - \nabla h(\x_{i-1}) \in \M\x_i$,
	and the convex and $L_h$-smooth properties of $h$,
	\ie, $\frac{1}{2L_h}||\nabla h(\x) - \nabla h(\y)||^2 \le D_h(\x,\y)$
\cite[Theroem 2.1.5]{nesterov:18}.
\end{proof}

\begin{remark}
\label{remark:function}
A worst-case rate of the BPP method,
in terms of the function value,
is studied in~\cite{nemirovski:04:pmw}. 
Under 
the monotone condition on $\M_\phi$
with $f(\u) = \delta_{\Us}(\u) := \begin{cases}0, &\text{if $\u\in\Us$,}\\ \infty, &\text{otherwise,}\end{cases}$
and $g(\vv) = \delta_{\Vs}(\vv)$,
where $\Us$ and $\Vs$ are compact sets,
the BPP method with a strongly convex $h$ 
(and thus the SA-MGDA) satisfies
\cite{nemirovski:04:pmw}:
\begin{align*}
    \max_{\vv\in\Vs}\phi(\bar{\u}_k,\vv)
    - \min_{\u\in\Us}\phi(\u,\bar{\vv}_k)
    \le \frac{\max_{\x\in\Us\times\Vs}D_h(\x,\x_0)}{k}
\end{align*}
\end{remark}
for the averaged iterates 
$\bar{\u}_k = \frac{1}{k}\sum_{i=1}^k\u_i$
and $\bar{\vv}_k = \frac{1}{k}\sum_{i=1}^k\vv_i$.
This, however, intrinsically cannot be generalized to the nonconvex-nonconcave case.

\begin{theorem}
    \label{thm:saag,strongmvi}
	Let $\M$~\eqref{eq:M} of the composite problem~\eqref{eq:P} 
	satisfy Assumption~\ref{assum:strongmvi} for $\mu\ge0$,
	and let $f,g$ and $\phi$ satisfy Assumption~\ref{assum:smooth}.
	Then, for $k\ge1$, $\tau\in\big(0,\frac{1}{L+\hat{L}}\big)$ 
	and for any $\x_*\in S_*^{\mu}(\M)$, 
	the sequence of the SA-MGDA (with $J=\infty$) satisfies
	\begin{align*}
	D_h(\x_*,\x_k) 
	\le
		\left(1 + \frac{2\tau\mu}{1+\tau L}\right)^{-k}
        D_h(\x_*,\x_{0})
	.\end{align*} 
\end{theorem}
\begin{proof}
The proof follows from Theorem~\ref{thm:bppm,str,conv}
		with $\mu_h>\hat{L}$.
\end{proof}

\subsection{
SA-MGDA with a finite $J$
}

In most of the practical cases, 
an exact maximization oracle (or equivalently, considering $J=\infty$) 
is computationally intractable.
This section thus studies the convergence behavior of the SA-MGDA
with a finite $J$,
under an additional bounded domain assumption
that includes the case in Remark~\ref{remark:function}.
\begin{assumption}
\label{assum:domain}
For some $\Omega\ge0$,
$||\x-\y||\le \Omega$ for all $\x,\y\in \dom f\times\dom g$.
\end{assumption}

The SA-MGDA with a finite $J$
can be viewed as an inexact variant of the BPP method \cite{eckstein:98:aii},
which generates a point $\x_{k+1}$ different from $\R\x_k$
at the $k$th iteration.
Therefore, the proof of the following theorem for the SA-MGDA 
with a finite $J$
first extends Theorem~\ref{thm:bppm,weakmvi} of the (exact) BPP
to its inexact variant in Appendix~\ref{appx:sa-mgda,weakmvi,inexact}.
Then, we consequently have the following result
for the SA-MGDA.

\begin{theorem}\label{thm:sa-mgda,weakmvi,inexact}
Let $\M$~\eqref{eq:M} of the composite problem~\eqref{eq:P} 
	satisfy Assumption~\ref{assum:weakmvi}
	for some $\rho\in\big[0,\frac{1}{2L+\hat{L}}\big)$, 
	and let $f,g$ and $\phi$ satisfy Assumptions~\ref{assum:smooth}
	and~\ref{assum:domain}
	for some $\Omega\ge0$.
	Then, the sequence $\{\x_k\}$
	of the SA-MGDA (with a finite $J$) 
	satisfies, for $k\ge1$, $\tau\in\big(\frac{\rho}{1-\rho L},\frac{1}{L+\hat{L}}\big)$
	and for any $\x_*\in X_*^\rho(\M)$,
\begin{align*}
&\min_{i=1,\ldots,k}\min_{s_i\in \M\R\x_{i-1}}\frac{||\s_i||^2}{2\left(\frac{1}{\tau}+L\right)}
\le \min_{i=1,\ldots,k}D_h(\R\x_{i-1}, 
    \x_{i-1}) \\
&\le \frac{D_h(\x_*,\x_0)}{\left(1 - \rho\left(\frac{1}{\tau}+L\right)\right)k}
+ \frac{3\Omega^2 \left(\frac{1}{\tau}+L\right)}{2\left(1 - \rho\left(\frac{1}{\tau}+L\right)\right)}\exp{-\frac{\frac{1}{\tau}-2L_{\vv\vv}}{2\left(\frac{1}{\tau}+2L_{\vv\vv}\right)}J}. 
\end{align*}
\end{theorem}
\begin{proof}
See Appendix~\ref{appx:sa-mgda,weakmvi,inexact}.
\end{proof}

This inequality reduces to that of
Theorem~\ref{thm:saag,weakmvi} as $J\to\infty$.
By Theorem~\ref{thm:sa-mgda,weakmvi,inexact}, 
the following corollary illustrates that
one can find an $\epsilon$-stationary point
using total $O(\epsilon^{-1}\log\epsilon^{-1})$
gradient computations in SA-MGDA.

\begin{corollary}
Under the conditions in Theorem~\ref{thm:sa-mgda,weakmvi,inexact},
the SA-MGDA method finds an $\epsilon$-stationary point,
i.e., a point $\x$ satisfying $\min_{\s\in \M\R\x}||\s||^2\le \epsilon$,
with $k=O(\epsilon^{-1})$ number of outer iterations and $J=O(\log(\epsilon^{-1}))$ number of inner iterations,
requiring total $O(\epsilon^{-1}\log \epsilon^{-1})$ gradient computations.
\end{corollary}

We similarly analyze the SA-MGDA with a finite $J$
under the strong MVI condition.
Note that the following theorem reduces to 
Theorem~\ref{thm:saag,strongmvi} as $J\to\infty$.

\begin{theorem}\label{thm:sa-mgda,strongmvi,inexact}
Let $\M$~\eqref{eq:M} of the composite problem~\eqref{eq:P} 
	satisfy Assumption~\ref{assum:strongmvi} for $\mu\ge0$,
	and let $f,g$ and $\phi$ satisfy Assumptions~\ref{assum:smooth}
	and~\ref{assum:domain}
	for some $\Omega\ge0$.
	Then, for $k\ge1$, $\tau\in\big(0,\frac{1}{L+\hat{L}}\big)$ 
	and for any $\x_*\in S_*^{\mu}(\M)$, 
	the sequence of the SA-MGDA (with finite $J$) satisfies
	\begin{align*}
    D_h(\x_*,\x_k)\le& 
    \left(
    1 + \frac{2\tau\mu}{1+\tau L}
    \right)^{-k} D_h(\x_*,\x_0)\\
    &+ \sum_{i=1}^k\left(
    1 + \frac{2\tau\mu}{1+\tau L}
    \right)^{-i+1}\frac{3\Omega^2 \left(\frac{1}{\tau}+L\right)}{2}\exp{-\frac{\frac{1}{\tau}-2L_{\vv\vv}}{2\left(\frac{1}{\tau}+2L_{\vv\vv}\right)}J}.
    \end{align*}
\end{theorem}
\begin{proof}
See Appendix~\ref{appx:sa-mgda,strongmvi,inexact}.
\end{proof}
\begin{corollary}
Under the conditions in Theorem~\ref{thm:sa-mgda,strongmvi,inexact},
the SA-MGDA method achieves 
$D_h(\x_*,\x_k)\le \epsilon$
with $k=O(\log(\epsilon^{-1}))$ number of outer iterations and $J=O(\log(\epsilon^{-1}))$ number of inner iterations,
requiring total $O(\log^2(\epsilon^{-1}))$ gradient computations.
\end{corollary}


\comment{
\subsection{Comparison to existing analysis}

Under Assumption~\ref{assum:weakmvi},
an extragradient-type method, named EG+~\cite{diakonikolas:21:emf},
finds a stationary point 
for $\rho \in \big[0,\frac{1}{4L}\big)$,
whereas the SA-MGDA method works for 
a wider region $\rho \in \big[0,\frac{1}{2L+\hat{L}}\big)$.
Under Assumption~\ref{assum:strongmvi},
	a variant of
	the dual extrapolation method~\cite{nesterov:07:dea}, 
	named OptDE~\cite{song:20:ode},
	has a linear rate.
In addition,
under a stronger assumption that
$\M$ is $\mu$-strongly monotone, 
\cite{mokhtari:20:aua} shows that the extragradient-type methods 
with step size $\frac{1}{4L}$
have rate $1-\frac{\mu}{4L}$,
while the SA-MGDA method with $\tau\approx\frac{1}{2L}$ (noting that $\hat{L}\le L$) has a faster rate
$\left(1+\frac{2\mu}{3L}\right)^{-1}$.
In short, the SA-MGDA method has faster rates compared to EG+ and OptDE under Assumption~\ref{assum:weakmvi} and Assumption~\ref{assum:strongmvi}, respectively; however, they are marginal advantages.
Here, our main contribution is the fact that the SA-MGDA method converges under settings that the MGDA method,
one of the fundamental minimax training methods, does not converge.
}

\subsection{Comparison to proximal point method}

Lemma~\ref{lem:single} and
Theorems~\ref{thm:bppm,weakmvi} and~\ref{thm:bppm,str,conv}
also apply
to the standard proximal point method 
(with 
$h(\x) = \frac{1}{2\tau}||\x||^2$),
\ie, $\x_{k+1} = (I + \tau\M)^{-1}\x_k$.
Such method, however, requires solving a regularized 
\emph{minimax} problem,
\begin{align*}
    & (\u_{k+1},\vv_{k+1}) \\ 
    =&\;
    \argmin{\u\in\reals^{d_u}}\max_{\vv\in\reals^{d_v}} \left\{f(\u) + \phi(\u,\vv) - g(\vv) + \frac{1}{2\tau}||\u-\u_k||^2 - \frac{1}{2\tau}||\vv-\vv_k||^2\right\},
\end{align*}
at each iteration, while the SA-MGDA needs one gradient descent update
and a \emph{maximization} at each iteration.
Both proximal point methods 
intrinsically have an implicit regularization (smoothing),
and thus have a good convergence guarantee,
while the latter is preferred
in terms of the computational complexity.

\section{Extensions of SA-MGDA}
\label{sec:exten}

\subsection{SA-MGDA with backtracking line-search}

The SA-MGDA method requires the knowledge of 
the global Lipschitz constants
of $\phi$, 
such as $L$, $\hat{L}$ and $L_{\vv\vv}$,
which can be locally conservative.
In addition, 
they are usually difficult to compute in practice.
%
To deal with these two drawbacks, 
we adapt a backtracking line-search technique
\cite{beck:09:afi,malitsky2018first,mukkamala2020convex}
in Algorithm~\ref{alg:sa-mgda,line},
which adjusts (decreases) 
$\tau$ at each iteration,
according to the local Lipschitz constant.
Note that the existence of $\R_{\M}^{\bar{h}}(\x_k)$ 
is guaranteed if
the maximization in $\vv$~\eqref{eq:uv}:
\begin{align*}
    \argmax{\vv\in\reals^{d_v}}  \Big\{
2\phi(\u_{k+1},\vv) - \frac{1}{2\bar{\tau}}||\vv - (\vv_k - \bar{\tau}\nabla_{\vv}\phi(\u_k,\vv_k))||^2
 - g(\vv)\Big\}
\end{align*}
is nonempty, including a local maximum.
Regarding the computation of $\R_{\M}^{\bar{h}}$ in Algorithm~\ref{alg:sa-mgda,line},
the standard SA-MGDA in Algorithm~\ref{alg:saag}
uses a proximal gradient ascent method on $\vv$ with known $L_{\vv\vv}$.
Here, one can apply its backtracking version 
in~\cite{beck:09:afi},
without the knowledge of $L_{\vv\vv}$.

\begin{algorithm}[h]
\caption{SA-MGDA with backtracking line-search 
    for $\tau_0\in\big[\frac{\delta}{L+\hat{L}},\infty\big)$, $\delta\in \big(2\rho(L+\hat{L}),1\big)$} 
\label{alg:sa-mgda,line}
\begin{algorithmic}
	\For{$k=0,1,\ldots$}
	\State{Find the smallest nonnegative integer $i_k$ such that
	$\bar{\x} = \R_{\M}^{\bar{h}}(\x_k)$ exists
	and satisfies}
	\State{$\frac{
		\bar{\tau}}{4}||\nabla\bar{h}(\bar{\x}) - \nabla\bar{h}(\x_k)||^2 \le D_{\bar{h}}(\bar{\x},\x_k), 
	\text{where $\bar{\tau} = \delta^{i_k}\tau_k$\; and\; 
	$\bar{h} = \frac{1}{2\bar{\tau}} ||\cdot||^2 - \phi$.}	
	$}
	\State{Let \;$\tau_{k+1} = \delta^{i_k}\tau_k$\; 
	and\; $h_{k+1} = \frac{1}{2\tau_{k+1}}||\cdot||^2 - \phi$.} 
	\State{$\x_{k+1} = \R_{\M}^{h_{k+1}}(\x_k)$}
	\EndFor
\end{algorithmic}
\end{algorithm}


Under an additional assumption that the 
sequence $\{\x_k\}$ of the SA-MGDA with backtracking line-search is bounded\footnote{
This holds, for example, 
under Assumption~\ref{assum:domain}, \ie,
when the domain of the problem is bounded}\!\!,
the following theorem 
shows that it
has an $O(1/k)$ worst-case convergence rate,
when the backtracking parameters $\tau_0$ and $\delta$
are appropriately chosen\footnote{
For instance, let $\tau_0$ be large enough, 
and $\delta = \frac{1}{2}$.
Then by Theorem~\ref{thm:sa-mgda,line,weakmvi}, the inequality $\min_{i=1,\ldots,k} D_{h_i}(\x_i,\x_{i-1}) 
\le\frac{D_{h_1}(\x_*,\x_0) + 3LD^2}{\left(1 - 6\rho L\right)k}$ 
holds, for the case
$\rho\in\big[0,\frac{1}{6L}\big)$,
and $\hat{L}\in \big[0,\frac{L}{2}$\big).
Note that the region of $\rho$ becomes small 
for an aggressive backtracking line-search,
\ie, a small $\delta$.}\!\!.

\begin{theorem}
\label{thm:sa-mgda,line,weakmvi}
Let $\M$~\eqref{eq:M} of the composite problem~\eqref{eq:P} 
satisfy Assumption~\ref{assum:weakmvi}
for some $\rho\in\big[0,\frac{1}{2(L+\hat{L})}\big)$, 
and let $f,g$ and $\phi$ satisfy Assumption~\ref{assum:smooth}. 
Additionally, 
assume that $\|\x_k-\x_*\|\le C$ for all $\x_k$ and for some $C\ge0$.
Then, the sequence $\{\x_k\}$ of the SA-MGDA with backtracking line-search for some 
$\tau_0\in \big[\frac{\delta}{L+\hat{L}},
\infty\big)$
and
$\delta\in \big(2\rho(L+\hat{L}),1\big)$ 
satisfies,
for any $\x_*\in X_*^\rho(\M)$,
\begin{align*}
\min_{i=1,\ldots,k} D_{h_i}(\x_i,\x_{i-1}) \le \frac{D_{h_1}(\x_*,\x_0) 
	+ \frac{\tilde{L}}{2}C^2}{\big(1 - \rho \tilde{L}\big)k},
\quad
\text{where } \tilde{L}:=\frac{2(L+\hat{L})}{\delta}.
\end{align*}
\end{theorem}
\begin{proof}
See Appendix~\ref{appx:sa-mgda,line,weakmvi}.
\end{proof}

\subsection{Non-Euclidean SA-MGDA for smooth adaptable problem}
\label{sec:nonlinear,bppm}

This section relaxes the smoothness condition on $\phi$
in Assumption~\ref{assum:smooth} 
to
its non-Euclidean extension, called
smooth adaptable condition~\cite{bolte:18:fom}. 

\subsubsection{Smooth adaptable problem and BPP}

We 
consider a function $\phi$ 
that is
smooth with respect to 
some Legendre function $\psi$,
\ie, both functions $L\psi - \phi$ and $L\psi + \phi$ are convex.
Examples of $\psi$ are
$\psi(\x) = -\sum_{i=1}^d\log x_i$~\cite{bauschke:16:adl}
and $\psi(\x) = \frac{1}{4}||\x||^4 + \frac{1}{2}||\x||^2$~\cite{bolte:18:fom}.

\begin{assumption}
\label{assum:smooth,adaptable}
A function $\phi\;:\;\reals^d\to \reals$ is continuously differentiable and $L$-smooth with respect to $\psi$, \ie,
there exists $L>0$ such that, for all $\x,\y\in\inter\dom\psi$,
\begin{align*}
|\inprod{\nabla\phi(\x) - \nabla\phi(\y)}{\x-\y}| \le L\inprod{\nabla\psi(\x) - \nabla\psi(\y)}{\x - \y}.
\end{align*}
\end{assumption}


Under Assumption~\ref{assum:smooth,adaptable},
we consider
\begin{align*}
h(\u,\vv) = \frac{1}{\tau}\psi(\u,\vv) - \phi(\u,\vv)
,\end{align*}
which is $\left(\frac{1}{\tau}-L\right)$-strongly convex\footnote{
We have that
$\inprod{\nabla h(\x) - \nabla h(\y)}{\x - \y}
= \Inprod{\Big(\frac{1}{\tau}\nabla\psi - \nabla\phi\Big)\x - \Big(\frac{1}{\tau}\nabla\psi - \nabla\phi\Big)\y}{\x-\y} 
= \frac{1}{\tau}\inprod{\nabla\psi(\x) - \nabla\psi(\y)}{\x-\y}  - \inprod{\nabla\phi(\x) - \nabla\phi(\y)}{\x-\y} 
\ge \left(\frac{1}{\tau} - L\right)\inprod{\nabla\psi(\x) - \nabla\psi(\y)}{\x-\y}$,
for all $\x,\y\in\reals^d$,
where the
inequality uses the smooth adaptable condition on $\phi$.} 
with respect to $\psi$,  
when $\frac{1}{\tau} > L$.
We further assume that 
$\M$
is weakly monotone with respect to $\psi$.

\begin{assumption}
\label{assum:M,weak,adaptable}
For some $\gamma\ge 0$, an operator $\M$ is $\gamma$-weakly monotone with respect to $\psi$, i.e.,
\begin{align*}
\inprod{\x-\y}{\w-\z} \ge -\gamma\inprod{\nabla\psi(\x) - \nabla\psi(\y)}{\x - \y},
\end{align*}
for all $\x,\y\in\inter\dom\psi$ and $(\x,\w),(\y,\z)\in\gra \M$. Also, it is maximal, \ie,
there exists no $\gamma$-weakly monotone operator,
with respect to $\psi$,
that its graph
properly contains $\gra \M$.
\end{assumption}

Then, under Assumption~\ref{assum:M,weak,adaptable},
the BPP update~\eqref{eq:bppm} 
with $\nabla h = \frac{1}{\tau}\nabla\psi - \nabla\phi$
is well-defined for $\frac{1}{\tau} - L > \gamma$.
\begin{lemma}
\label{lem:nonlinear,single}
Let $\M$ satisfy Assumption~\ref{assum:M,weak,adaptable} 
for some $\psi$ and $\gamma\ge0$,
and $h$ 
be a Legendre function and $\mu_h$-strongly convex with respect to $\psi$, \ie,
\begin{align*}
\inprod{\nabla h(\x) - \nabla h(\y)}{\x - \y} 
\ge \mu_h\inprod{\nabla\psi(\x) - \nabla\psi(\y)}{\x - \y},
\end{align*}
for all $\x,\y\in\inter\dom \psi$.
Then, if $\mu_h > \gamma$, 
the $h$-resolvent $\R$ is single-valued on $\inter\dom h$. 
\end{lemma}
\begin{proof}
	Note that $\nabla h + \M = (\nabla h -\gamma\nabla\psi) + (\M+\gamma\nabla\psi)$. 
	From the condition that $\mu_h > \gamma$ and $\M$ is $\gamma$-weakly monotone
	with respect to $\psi$, 
	it is straightforward to show that $h-\gamma\psi$ is a Legendre function, 
	and $\M+\gamma\nabla\psi$ is maximally monotone. Then Corollary 2.3 in \cite{bauschke:10:grf} shows that $\ran(\nabla h + \M) = \reals^d$. 
	This implies that $\R \x$ is nonempty for all $\x\in\inter\dom h$.
	Assume that $\y,\z \in \R\x$.
	Since 
	$\nabla h(\x)-\nabla h(\y) \in \M\w$
	and $\nabla h(\x) - \nabla h(\z) \in \M\z$, we have the following inequality:
	\begin{align*}
	-\gamma\inprod{\nabla \psi(\y) - \nabla\psi(\z)}{\y-\z}
	&\le -\inprod{\nabla h(\y)-\nabla h(\z)}{\y-\z} \\
	&\le - \mu_h \inprod{\nabla \psi(\y) - \nabla\psi(\z)}{\y-\z}.
	\end{align*}
	So if $\mu_h > \gamma$,
	the inequality implies that $\y = \z$.
\end{proof}

\subsubsection{Non-Euclidean SA-MGDA method from BPP}

The resulting BPP method for smooth adaptable problems, 
named the non-Euclidean SA-MGDA,
uses mirror descent steps~\cite{nemirovski:83}
with respect to $\psi$ (see also~\cite{bauschke:16:adl,bolte:18:fom}),
instead of the gradient steps in the standard SA-MGDA.
This simplifies, for a separable $\psi(\u,\vv) = \psi_{\u}(\u) + \psi_{\vv}(\vv)$, as below, 
which reduces to the standard SA-MGDA~\eqref{eq:uv} 
when $\psi(\x) = \frac{1}{2}||\x||^2$:
\begin{align}
			    \u_{k+1} &= \argmin{\u\in 
			    \reals^{d_u}}
			    \Big\{\phi(\u_k,\vv_k) + 
\inprod{\nabla_{\u}\phi(\u_k,\vv_k)}{\u - \u_k} + \frac{1}{\tau}D_{\psi_{\u}}(\u,\u_k) + f(\u)\Big\}
			    \label{eq:ne_samgda}
			    \\
\vv_{k+1} &= \argmax{\vv\in 
    \reals^{d_v}} 
    \Big\{2\phi(\u_{k+1},\vv) \\
&\hspace{55pt}    
-    \phi(\u_k,\vv_k)
- \inprod{\nabla_{\vv}\phi(\u_k,\vv_k)}{\vv- \vv_k} - \frac{1}{\tau}D_{\psi_{\vv}}(\vv,\vv_k) 
- g(\vv)\Big\} \nonumber \\
&= \argmax{\vv\in 
\reals^{d_v}}  
\Big\{
2\phi(\u_{k+1},\vv) - \frac{1}{\tau}D_{\psi_{\vv}}(\vv,\tilde{\vv}_k)
 - g(\vv)\Big\}, \nonumber
\end{align}
where
$\tilde{\vv}_k := \nabla\psi_{\vv}^*(\nabla\psi_{\vv}(\vv_k) - \tau\nabla_{\vv}\phi(\u_k,\vv_k))$
is an \emph{anchor} point
and $\psi_{\vv}^*$ is the conjugate of $\psi_{\vv}$.
Note that the analogous standard non-Euclidean MGDA
that considers $\vv_k$, instead of $\tilde{\vv}_k$, in~\eqref{eq:ne_samgda}
does not have a convergence guarantee,
unlike this method under the considered setting.
The minimization in $\u$ can be solved by one proximal mirror descent update
(see~\cite{bauschke:16:adl,bolte:18:fom}),
while the maximization in $\vv$ needs multiple steps
of a proximal mirror descent update
under our underlying assumption that $\phi(\u,\vv)$ is relative smooth in $\vv$ with respect to $\psi_{\vv}$.

This method has a rate $O(1/k)$
under the MVI condition, \ie, Assumption~\ref{assum:weakmvi} with $\rho=0$. 
%
\begin{theorem}
Let $\M$~\eqref{eq:M} of the composite problem~\eqref{eq:P} 
satisfy Assumption~\ref{assum:weakmvi} for $\rho=0$
and Assumption~\ref{assum:M,weak,adaptable} for some $\gamma\ge0$, 
and let $\phi$ satisfy Assumption~\ref{assum:smooth,adaptable}.
Then, the sequence $\{\x_k\}$ 
of the non-Euclidean SA-MGDA (with $J=\infty$)
satisfies, for $k\ge1$, $\tau\in\big(0,\frac{1}{L+\gamma}\big)$
and for any $\x_*\in X_*^0(\M)$,
$ 
\min_{i=1,\ldots,k} D_h(\x_i,\x_{i-1})
\le \frac{D_h(\x_*,\x_0)}{k}
.$ 
\end{theorem}
\begin{proof}
By Lemma~\ref{lem:nonlinear,single}, $\mu_h = \frac{1}{\tau} - L > \gamma$
implies that $\R\x$ exists for any $\x$.
The proof of Lemma~\ref{lem:expansive,weakmvi} 
also works for $\rho=0$ and any Legendre function $h$, \ie,
$ 
0 \le D_h(\x_*,\x_i) - D_h(\x_*,\x_{i+1}) - D_h(\x_{i+1},\x_i)
$ 
for all $i\ge0$.
Then, 
based on 
the proof of Theorem~\ref{thm:bppm,weakmvi}.
we have
$ 
\min_{i=1,\ldots,k}D_h(\x_i,\x_{i-1})
\le \frac{1}{k}\sum_{i=1}^k D_h(\x_i,\x_{i-1})
\le \frac{D_h(\x_*,\x_0)}{k}
$. 
\end{proof}

\comment{\cred
\subsubsection{Application to Poisson linear inverse problems on a simplex}

Let us consider a Poisson linear inverse problem~\cite{bauschke:16:adl} on the simplex domain:
    \begin{align*}
        \min_{\u\in\Delta^n} \sum_{j=1}^m \Big\{b_j \log \frac{b_j}{(A\u)_j} + (A\u)_j - b_j\Big\},
    \end{align*}
    where $\Delta^n:=\{\u=(u_1,\ldots,u_n)\in\reals_+^n:\sum_{j=1}^n u_j = 1\}$, $A\in\reals_+^{m\times n}$, and $b\in\reals_{++}^m$. This example can be rewritten as the following minimax problem.
    \begin{align}\label{ex:plip}
        \min_{\u\in\reals_+^n}\max_{\vv\in\reals} \phi(\u,v):= \sum_{j=1}^m \Big\{b_j \log \frac{b_j}{(A\u)_j} + (A\u)_j - b_j\Big\} + \vv\Big(\sum_{j=1}^n u_j - 1\Big).
    \end{align}
    Then, $\phi$ is $L$-smooth with respect to $\psi(\u,\vv) := -\sum_{j=1}^n \log u_j + \frac{1}{2}\vv^2$ for $L=\sum_{j=1}^m b_j$ on its domain $(\u,\vv)\in \reals_+^n\times\reals$.

Let us solve \eqref{ex:plip} using the non-Euclidean SA-MGDA (the BPP method with $h= \frac{1}{\tau}\psi - \phi$) in~\eqref{eq:ne_samgda}. 
Then one can easily get that $\nabla_{\u}\psi(\u,\vv) = -\u^{-1}$, $\nabla_{\vv} \psi(\u,\vv) = \vv$, $\nabla_{\u}\phi(\u,\vv) = \sum_{j=1}^n\Big(1-\frac{b_j}{\inprod{a_j}{\u}}\Big)a_j + \vv e$, and $\nabla_{\vv}\phi(\u,\vv)= \inprod{e}{\u} - 1$, where $\u^{-1}:=(u_1^{-1},\ldots,u_n^{-1})$, $a_j$ is the $j$-th row vector of $A$, and $e$ is a vector with all entries equal to $1$.
Thus, the corresponding non-Euclidean SA-MGDA is
\begin{align*}
    \u_{i+1} &= \Big[\u_i^{-1} + \tau\Big(\sum_{j=1}^n\Big(1-\frac{b_j}{\inprod{a_j}{\u_i}}\Big)a_j + \vv_i e\Big)\Big]^{-1}\\
    \vv_{i+1} &= \vv_i + \tau\Big(2(\inprod{e}{\u_{i+1}} - 1) - (\inprod{e}{\u_i} - 1)\Big).
\end{align*}
}

\section{Numerical results}
\label{sec:result}
This section compares our SA-MGDA method
with the MGDA~\cite{nouiehed:19:sac}
on two toy experiments
and one realistic fair training experiment\footnote{The code is available at \url{https://github.com/csfh1379/sa-mgda}.}\!\!.


\subsection{Toy examples}
We first consider 
\begin{align*}
\phi(u,v) = -\frac{\rho L^2}{4}u^2 + L\sqrt{1-\frac{\rho^2L^2}{4}}uv + \frac{\rho L^2}{4}v^2
,\end{align*}
with $f=g=0$.
Its saddle-subdifferential $\M_\phi$ satisfies the $L$-Lipschitz continuity and the $\rho$-weak MVI condition, where $L=1$ and $\rho=\frac{1}{4}$.
The problem is nonconcave on $\vv$, 
so directly applying MGDA is not guaranteed to work.
For a practical comparison,
we apply the MGDA~\cite{nouiehed:19:sac} 
to a regularized function
$\phi(u,v) -\frac{\lambda}{2}(v-v_0)^2$ 
with $\lambda > L_{vv} = \frac{\rho L^2}{2} = \frac{1}{8}$,
so that it becomes strongly concave on $v$.
(This stems from the strategy taken in~\cite{nouiehed:19:sac}
when the function is concave in $v$.)
Figure~\ref{fig:first_toy_example} 
illustrates that SA-MGDA (with $\lambda=0$)
outperforms 
the MGDA~\cite{nouiehed:19:sac} with $\lambda=1,\; \frac{1}{2},\; \frac{1}{4}$, 
for $(u_0,v_0)=(0,1)$.
Here, the maximization on $v$ is computed exactly
for all methods.

Next, we consider another toy example in \cite[Example~3]{pethick:22:elc}:
\begin{align*}
    \M_\phi(u,v) = (\psi(u,v)-v,\psi(v,u)+u),
\end{align*}
where 
$\psi(u,v)=\frac{1}{8}u(-1 + u^2 + v^2)(-1+4u^2+4v^2)$,
with $f(u) = \delta_{\{x\;:\;|x|\le1.1\}}(u)$
and $g(v) = \delta_{\{x\;:\;|x|\le1.1\}}(v)$.
The operator $\M_{\phi}$ satisfies $L$-Lipschitz continuity 
(within the domain)
and $\rho$-weak MVI condition,
where $\rho<\frac{1}{4L}$\footnote{$L=\frac{\sqrt{5449798437 - 173756\sqrt{890712929}}}{10000}\approx1.6251$ and $\rho = \frac{2304}{16465}\approx0.1399$.}\!\!.
For a practical comparison,
we apply the MGDA to a regularized operator
$\M_\phi(u,v) + \lambda(0,v-v_0)$.
Figure~\ref{fig:second_toy_example} 
presents that the SA-MGDA (with 
$\lambda=0$)
outperforms 
the MGDA~\cite{nouiehed:19:sac} with $\lambda=0,\; \frac{L}{4},\; L$. 
Here, 
the maximization on $v$ is computed inexactly
with
$J=10$ number of projected gradient ascent steps
for all methods. 

\subsection{Fair classification}\label{sec:fair}

To ensure that the trained model is fair to all categories, 
\cite{mohri:19:afl} considered a minimax problem 
that minimizes the maximum loss among the categories.
We study such fair classification experiment
in \cite{nouiehed:19:sac}
on the Fashion MNIST data set\footnote{
This consists of 28$\times$28 grayscale cloth images of ten categories;
60000 data for training and 10000 for test.}
\cite{xiao:17:fma}.
Similar to \cite{mohri:19:afl,nouiehed:19:sac}, 
we focus on the data labeled as T-shirt/top, Coat, and Shirt. 
The corresponding minimax problem is 
\begin{align*}
\min_{\u} \max_{i=1,2,3} \mathcal{L}_i(\u)
,\end{align*}
where $\u$ 
denotes the parameters of the neural network
(see Appendix~\ref{appx:network} for the 
details), 
and $\mathcal{L}_1$, $\mathcal{L}_2$, and $\mathcal{L}_3$ 
denote the cross-entropy losses\footnote{
Suppose $i$ has $n$
data $\{\x_{i,j}\}_{j=1}^n$.
Then, the corresponding cross-entropy loss is defined as
$\mathcal{L}_i(\u) = -\frac{1}{n}\sum_{j=1}^n
\log f_{\u}^{(i)}(\x_{i,j})$,
where $f_{\u}^{(i)}$ is the $i$-th entry 
of the neural network $f_{\u}$
that learns the probability
of the data to be in the class $i$.
}
of the training data in each category, respectively.
This is equivalent to the problem\footnote{ 
The problem 
is highly nonconvex and nonsmooth on $\w$, 
while it is concave on $\ttt$.
Therefore, the purpose of this experiment is to investigate whether or not our theoretical understanding
expands to real-world problems.}
\begin{align*}
\min_{\u} \max_{\vv\in \Vs}\sum_{i=1}^3 v_i \mathcal{L}_i(\u)
,\end{align*}
where 
$\Vs=\big\{\vv\in\reals_+^3\;:\;
\sum_{i=1}^3 v_i = 1\big\}$,
\ie, $\phi(\u,\vv) = \sum_{i=1}^3 v_i \mathcal{L}_i(\u)$ with $f=0$ and $g(\vv)=\delta_{\Vs}(\vv)$.
Since the problem 
is not strongly concave in $\vv$, 
\cite{nouiehed:19:sac} applied the MGDA to a regularized problem
$ 
\min_{\u} \max_{\vv\in \Vs}\sum_{i=1}^3 v_i \mathcal{L}_i(\u) - \frac{\lambda}{2}\sum_{i=1}^3 v_i^2
$, 
where $\lambda$ is a positive regularization parameter.

\begin{figure}[t!]
     \centering
     \begin{subfigure}[b]{0.48\textwidth}
         \centering
         \includegraphics[width=\textwidth]{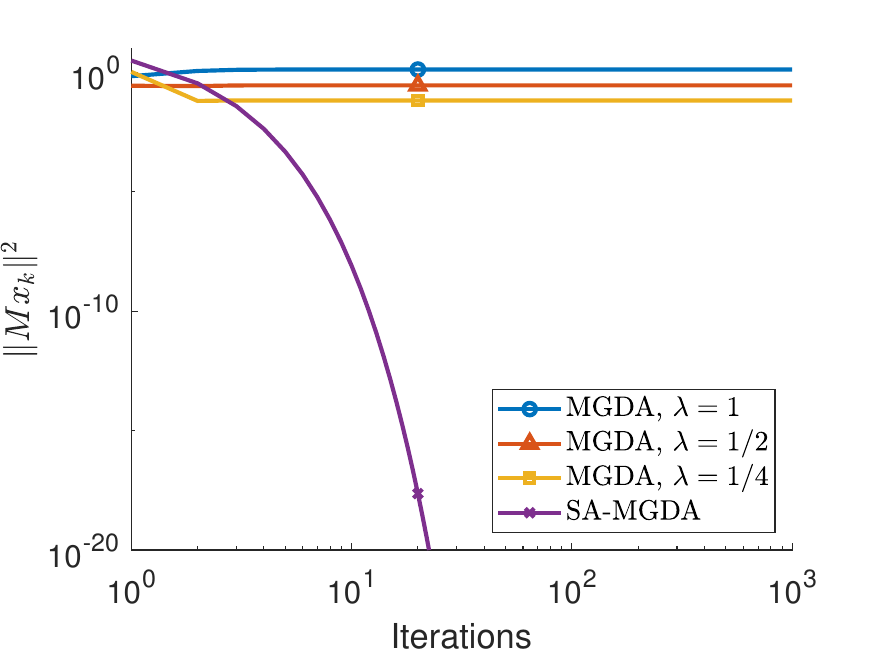}
         \caption{First toy example}
         \label{fig:first_toy_example}
     \end{subfigure}
     \hfill
     \begin{subfigure}[b]{0.48\textwidth}
         \centering
         \includegraphics[width=\textwidth]{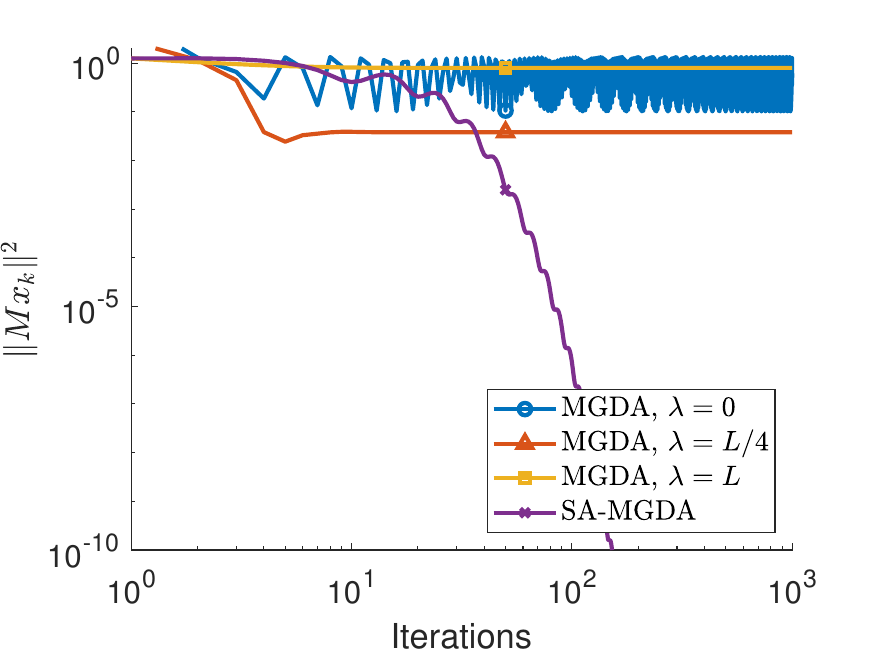}
         \caption{Second toy example}
         \label{fig:second_toy_example}
     \end{subfigure}
        \caption{Toy examples}
        \label{fig:result}
\end{figure}


\begin{figure}[ht!]
	\centering
	\includegraphics[width=.48\columnwidth]{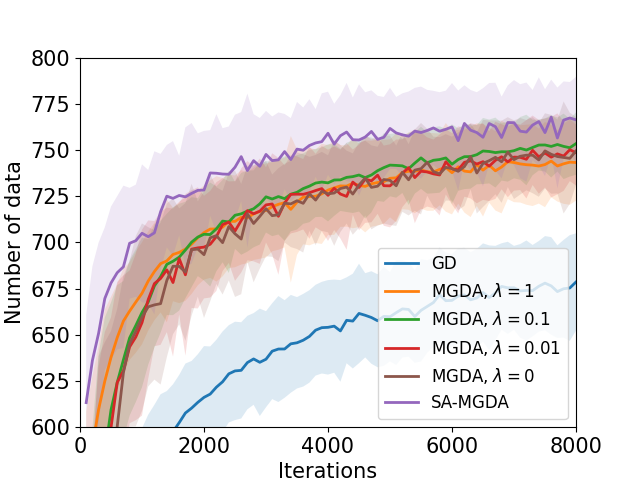}
	\includegraphics[width=.48\columnwidth]{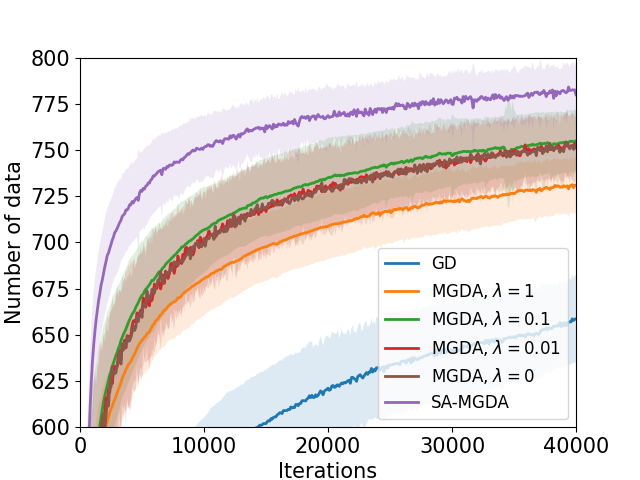}
	\caption{Fair classification: the number of correctly classified test data
		for the worst category vs. iteration. (Left) $\tau=0.01$, (Right) $\tau = 0.001$}
	\label{fig:fair}
\end{figure}

We ran 8000 iterations of the MGDA and SA-MGDA methods 
with the learning rate $\tau = 0.01$,
where the maximization on $\vv$ 
is computed exactly.
For the MGDA, we considered various regularization parameters $\lambda = 0, 0.01, 0.1, 1$.
We also ran a normal training,
$\min_{\u}\sum_{i=1}^3\frac{1}{3}\mathcal{L}_i(\u)$,
using a gradient descent (GD) method on $\u$ 
for comparison.
We performed 50 independent simulations for each case,
and, in Table \ref{table:fmnist}, we report the mean and standard deviation
of the number of correctly classified test data (out of 1000) 
for each category and the worst\footnote{
The worst category denotes the smallest number of correctly classified test data among the three categories.}
category.
Fig.~\ref{fig:result} plots 
the results of
the worst category versus iterations.

Table \ref{table:fmnist} presents that
the model learned by the SA-MGDA method has the best performance,
in terms of the worst category, yielding the most fairness,
even without an explicit regularization and its parameter tuning.
%
In addition, the learned model with the SA-MGDA method has the best accuracy for two categories
and the second best 
for one category among fair trainings.

\begingroup
\begin{table*}[h!]
\setlength{\tabcolsep}{4.0pt}
\renewcommand*{\arraystretch}{0.7}
\centering
\caption{Fair classification: the mean and standard deviation of the number of correctly classified test data for normal and fair trainings
($\tau = 0.01$).
}
\label{table:fmnist}
\begin{tabular}{c c c| cc cc cc |cc}
	\toprule[1pt]
	& \multirow{2}{*}{Method} & \multirow{2}{*}{$\lambda$} & \multicolumn{2}{c}{T-shirt/top} & \multicolumn{2}{c}{Coat} & \multicolumn{2}{c|}{Shirt} & \multicolumn{2}{c}{\textbf{Worst}} \\
	& & & mean & std & mean & std & mean & std & mean & std\\
	\midrule[1pt]
	Normal & GD & $\cdot$ & 852.5 & 12.3 & 855.5 & 18.4 & 678.6 & 26.3 & 678.6 & 26.3\\
	\midrule[0.6pt]
	\multirow{5}{*}{Fair} 
	& \multirow{4}{*}{MGDA} & $1$ & \textbf{812.7} & 16.8 & 813.0 & 25.6 & 743.8 & 21.4 & 743.2 & 20.6\\
	& & $0.1$ & 783.4 & 20.3 & 781.4 & 33.0 & 770.3 & 21.8 & 753.5 & 16.9\\
	& & $0.01$ & 778.3 & 28.0 & 772.0 & 29.4 & 773.4 & 25.6 & 748.5 & 17.7\\
	& & $0$ & 777.6 & 27.8 & 778.2 & 35.3 & 772.1 & 21.4 & 749.5 & 18.3\\
    \cmidrule[0.4pt]{2-11}
	& SA-MGDA & $\cdot$ & 807.9 & 33.2 & \textbf{822.2} & 29.7 & \textbf{781.7} & 33.1 & \textbf{766.4} & 23.2\\
	\bottomrule[1pt]
\end{tabular}
\end{table*}
\endgroup

We present an additional numerical result with smaller learning rate, $\tau=0.001$. The other settings of experiment are equivalent to 
the case $\tau=0.01$,
except that we ran 40000 iterations for each case.
\begingroup
\setlength{\tabcolsep}{4.0pt}
\renewcommand*{\arraystretch}{0.7}
\begin{table*}[h!]
\centering
	\caption{Fair classification: the mean and standard deviation of the number of correctly classified test data for normal and fair trainings. ($\tau = 0.001$)
	}
	\label{table:fmnist_0.001}
	\begin{tabular}{c c c| cc cc cc |cc}
		\toprule[1pt]
		& \multirow{2}{*}{Method} & \multirow{2}{*}{$\lambda$} & \multicolumn{2}{c}{T-shirt/top} & \multicolumn{2}{c}{Coat} & \multicolumn{2}{c|}{Shirt} & \multicolumn{2}{c}{\textbf{Worst}} \\
		& & & mean & std & mean & std & mean & std & mean & std\\
		\midrule[1pt]
		Normal & GD & $\cdot$ & 849.2 & 9.8 & 843.2 & 17.8 & 658.6 & 23.5 & 658.6 & 23.5\\
		\midrule[0.6pt]
		\multirow{5}{*}{Fair} & \multirow{4}{*}{MGDA} & $1$ & \textbf{805.8} & 14.3 & 797.2 & 20.2 & 731.0 & 14.8 & 731.0 & 14.8\\
		& & $0.1$ & 780.0 & 12.8 & 766.5 & 23.0 & 763.4 & 15.7 & 755.0 & 16.4\\
		& & $0.01$ & 772.0 & 15.6 & 761.3 & 21.1 & 767.7 & 14.5 & 753.0 & 14.7\\
		& & $0$ & 775.2 & 14.9 & 761.4 & 24.0 & 767.0 & 13.0 & 752.9 & 18.9\\
		\cmidrule[0.4pt]{2-11}
		& SA-MGDA & $\cdot$ & 794.5 & 17.4 & \textbf{805.8} & 21.1 & \textbf{789.9} & 18.6 & \textbf{779.9} & 18.0\\
		\bottomrule[1pt]
	\end{tabular}
\end{table*}
\endgroup

\comment{
\begin{figure}[h]
\centering
\includegraphics[width=.7\columnwidth]{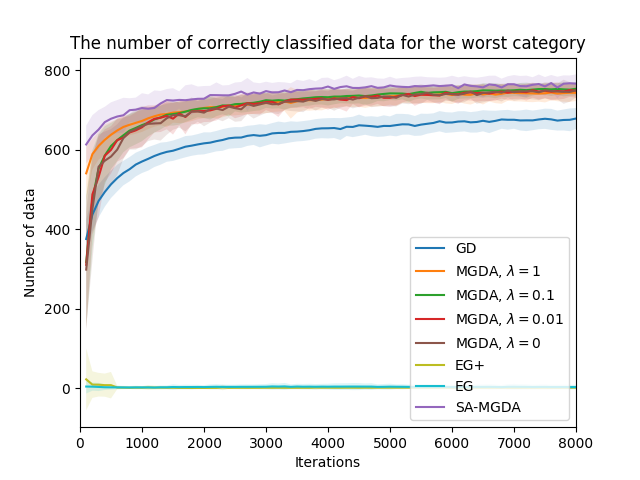}
\caption{The number of correctly classified test data
for the worst category vs. iteration. ($\tau = 0.01$)}
\end{figure}

\begin{figure}[h]
\centering
\includegraphics[width=.7\columnwidth]{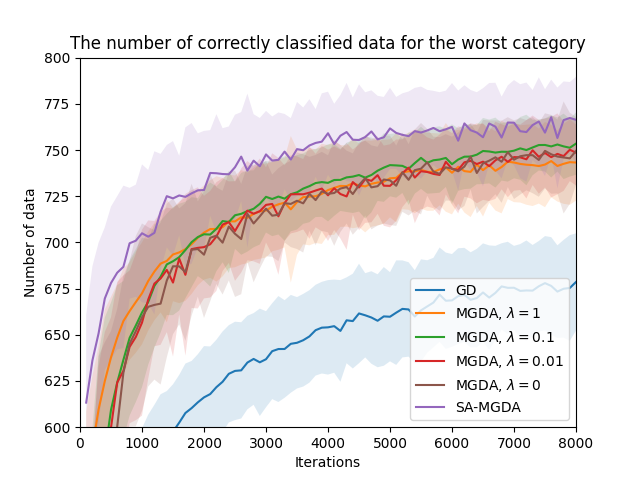}
\caption{The number of correctly classified test data
for the worst category vs. iteration. ($\tau = 0.01$)}
\label{fig:rate}
\end{figure}
}



\section{Conclusion}
\label{sec:conc}

This paper proposed a semi-anchoring approach
to a multi-step gradient descent ascent method
for structured nonconvex-nonconcave smooth minimax problems.
This is a new instance of the Bregman proximal point method 
for operators under the weak MVI condition.
%
We showed that the proposed method guarantees convergence
without regularization and its parameter tuning,
unlike the smoothing technique.
We further studied its backtracking line-search version,
and its non-Euclidean version for smooth adaptable functions.
Numerical experiments suggest that the proposed approach has potential to improve
training dynamics of real-world minimax problems.
We leave extending this work
to a more general stochastic nonconvex-nonconcave setting as future work.

\comment{
\begin{ack}
Use unnumbered first level headings for the acknowledgments. All acknowledgments
go at the end of the paper before the list of references. Moreover, you are required to declare
funding (financial activities supporting the submitted work) and competing interests (related financial activities outside the submitted work).
More information about this disclosure can be found at: \url{https://neurips.cc/Conferences/2021/PaperInformation/FundingDisclosure}.

Do {\bf not} include this section in the anonymized submission, only in the final paper. You can use the \texttt{ack} environment provided in the style file to autmoatically hide this section in the anonymized submission.
\end{ack}
}

\appendix

\section{Proof of Lipschitz continuity of $\nabla \phi$ and ${\bf M}_\phi$}
\label{appx:lipschitz}
Let $\x:=(\u,\vv)$ and $\y:=(\bar{\u},\bar{\vv})$. 
Since $\|\nabla \phi(\u,\vv)-\nabla \phi(\bar{\u},\bar{\vv})\|=\|\M_\phi\x-\M_\phi\y\|$, 
it is enough to show that 
there exists a constant $L>0$ such that
$\|\nabla \phi(\u,\vv) - \nabla \phi(\bar{\u},\bar{\vv})\|\le L\|\x-\y\|$. 

By Assumption~\ref{assum:smooth}, we have the following bounds
\begin{align*}
\|\nabla \phi(\u,\vv)-\nabla \phi(\bar{\u},\vv)\|^2
&= \|\nabla_{\u} \phi(\u,\vv)-\nabla_{\u} \phi(\bar{\u},\vv)\|^2 + \|\nabla_{\vv} \phi(\u,\vv)-\nabla_{\vv} \phi(\bar{\u},\vv)\|^2 \\
&
\le (L_{\u\u}^2+L_{\vv\u}^2)\|\u-\bar{\u}\|^2, \\
\|\nabla \phi(\bar{\u},\vv)-\nabla \phi(\bar{\u},\bar{\vv})\|^2
&= \|\nabla_{\u} \phi(\bar{\u},\vv)-\nabla_{\u} \phi(\bar{\u},\bar{\vv})\|^2 + \|\nabla_{\vv} \phi(\bar{\u},\vv)-\nabla_{\vv} \phi(\bar{\u},\bar{\vv})\|^2 \\
&
\le (L_{\u\vv}^2+L_{\vv\vv}^2)\|\vv-\bar{\vv}\|^2.
\end{align*}
Then, we can show the Lipschitz continuity as below.
\begin{align*}
\|\nabla \phi(\u,\vv)-\nabla \phi(\bar{\u},\bar{\vv})\| 
&\le \|\nabla \phi(\u,\vv)-\nabla \phi(\bar{\u},\vv)\| + \|\nabla \phi(\bar{\u},\vv)-\nabla \phi(\bar{\u},\bar{\vv})\| \\
&\le \sqrt{L_{\u\u}^2+L_{\vv\u}^2}\|\u-\bar{\u}\|+\sqrt{L_{\u\vv}^2+L_{\vv\vv}^2}\|\vv-\bar{\vv}\|\\
&\le \sqrt{L_{\u\u}^2+L_{\vv\u}^2+L_{\u\vv}^2+L_{\vv\vv}^2}\sqrt{\|\u-\bar{\u}\|^2+\|\vv-\bar{\vv}\|^2}\\
&\qquad \qquad(\because \text{Cauchy-Schwarz  inequality})\\
&= \sqrt{L_{\u\u}^2+L_{\vv\u}^2+L_{\u\vv}^2+L_{\vv\vv}^2}\|\x-\y\|.
\end{align*}
\qed

\section{Proof of weak monotonicity of ${\bf M}$}
\label{appx:weak}
By Assumption~\ref{assum:smooth},
$\phi(\cdot,\vv)$ is $L_{\u\u}$-weakly convex for fixed $\vv$,
and $-\phi(\u,\cdot)$ is $L_{\vv\vv}$-weakly convex for fixed $\u$.
Then, using the weak convexity on $\u$, we have
\begin{align*}
\phi(\bar{\u},\vv) &\ge \phi(\u,\vv) + \inprod{\nabla_{\u}\phi(\u,\vv)}{\bar{\u} - \u}
- \frac{L_{\u\u}}{2}||\bar{\u} - \u||^2, \\
\phi(\u,\bar{\vv}) &\ge \phi(\bar{\u},\bar{\vv}) + \inprod{\nabla_{\u}\phi(\bar{\u},\bar{\vv})}{\u - \bar{\u}}
- \frac{L_{\u\u}}{2}||\u - \bar{\u}||^2,
\end{align*}
for all $\u,\bar{\u}\in\reals^{d_u}$ and $\vv,\bar{\vv}\in\reals^{d_v}$.
Similarly, using the weak convexity on $\vv$, we have
\begin{align*}
-\phi(\u,\bar{\vv}) &\ge -\phi(\u,\vv) - \inprod{\nabla_{\vv}\phi(\u,\vv)}{\bar{\vv} - \vv}
- \frac{L_{\vv\vv}}{2}||\bar{\vv} - \vv||^2, \\
-\phi(\bar{\u},\vv) &\ge -\phi(\bar{\u},\bar{\vv}) - \inprod{\nabla_{\vv}\phi(\bar{\u},\bar{\vv})}{\vv - \bar{\vv}}
- \frac{L_{\vv\vv}}{2}||\vv - \bar{\vv}||^2
.\end{align*}
%
Let $\x=(\u,\vv)$, $\y=(\bar{\u},\bar{\vv})$, $\w_{\phi}=(\nabla_{\u}\phi(\u,\vv),-\nabla_{\vv}\phi(\u,\vv))$,
and $\z_{\phi}=(\nabla_{\u}\phi(\bar{\u},\bar{\vv}),$ 
$-\nabla_{\vv}\phi(\bar{\u},\bar{\vv}))$.
Then, 
summing the above four inequalities yields
\begin{align*}
\inprod{\x - \y}{\w - \z} 
&\ge \inprod{\x - \y}{\w_{\phi} - \z_{\phi}}
\ge -L_{\u\u}||\u-\bar{\u}||^2 - L_{\vv\vv}||\vv - \bar{\vv}||^2 \\
&\ge -\max\{L_{\u\u},L_{\vv\vv}\}||\x - \y||^2,
\end{align*}
for all $(\x,\w),(\y,\z)\in\gra\M$,
where the first inequality uses the convexity of $f$ and $g$.
\qed


\section{Proof of Proposition~\ref{prop:linear}}
\label{appx:linear}
When $f=g=0$, the sequence $\{\u_k,\vv_k\}$ generated by the PDHG 
has a relationship
\begin{align*}
\u_{k+1} - 2\u_k +\u_{k-1} = -\tau\B(\vv_k - \vv_{k-1})
= -\tau^2\B\B^\top(2\u_k - \u_{k-1})
.\end{align*}
This can be written in a matrix form
\begin{align*}
\left[\begin{array}{c}
\u_{k+1} \\
\u_k
\end{array}\right]
= \underbrace{\left[\begin{array}{cc}
	2(\I - \tau^2\B\B^\top) & -(\I - \tau^2\B\B^\top) \\
	\I & \zero
	\end{array}\right]}_{\T_{\u}}
\left[\begin{array}{c}
\u_k \\
\u_{k-1}
\end{array}\right]
.\end{align*}
The spectral radius of $\T_{\u}$, denoted by $\rho(\T_{\u})$,
determines the convergence rate of the PDHG.
In specific, for any $\epsilon > 0$,
there exists $K\ge0$ such that 
$[\rho(\T_{\u})]^k \le ||\T_{\u}^k|| \le [\rho(\T_{\u}) + \epsilon]^k$ 
for all $k\ge K$,
and this yields
\begin{align*}
||\hat{\u}_{k+1} - \hat{\u}_*||^2 \le (\rho(\T_{\u}) + \epsilon)^{2k}||\hat{\u}_1 - \hat{\u}_*||^2
,\end{align*}
where $\hat{\u}_k := (\u_k^\top\;\u_{k-1}^\top)^\top$
and $\hat{\u}_* := (\u_*^\top\;\u_*^\top)^\top = \zero$.


Considering the SVD factorization of $\B = \U\Sigma\V^\top$, 
the spectral radius of $\T_{\u}$ can be rewritten as
\begin{align*}
\rho(\T_{\u}) = \max_{\sigma_{\min}(\B)\le\sigma\le\sigma_{\max}(\B)} \rho(\T_{\u,\sigma}),
\quad\text{where }
\T_{\u,\sigma} := \left[\begin{array}{cc}
2(1 - \tau^2\sigma^2) & -(1 - \tau^2\sigma^2) \\
1 & 0
\end{array}\right]
.\end{align*}
The matrix
$\T_{\u,\sigma}$
has two complex eigenvalues $1-\tau^2\sigma^2 \pm i\sqrt{1-\tau^2\sigma^2 - (1-\tau^2\sigma^2)^2}$
with magnitude $\sqrt{1-\tau^2\sigma^2}$.
Therefore, we have $\rho(\T_{\u}) = \sqrt{1 - \tau^2\sigma_{\min}^2(\B)}$ and
\begin{align*}
||\hat{\u}_{k+1} - \hat{\u}_*||^2 
\le \left(\sqrt{1 - \tau^2\sigma_{\min}^2(\B)} + \epsilon\right)^{2k}||\hat{\u}_1 - \hat{\u}_*||^2
.\end{align*}

Similarly, the sequence $\{\vv_k\}$ generated by the PDHG has a relationship
\begin{align*}
\left[\begin{array}{c}
\vv_{k+1} \\
\vv_k
\end{array}\right]
= \left[\begin{array}{cc}
2(\I - \tau^2\B^\top\B) & -(\I - \tau^2\B^\top\B) \\
\I & \zero
\end{array}\right]
\left[\begin{array}{c}
\vv_k \\
\vv_{k-1}
\end{array}\right]
,\end{align*}
which then yields
\begin{align*}
||\hat{\vv}_{k+1} - \hat{\vv}_*||^2 
\le \left(\sqrt{1 - \tau^2\sigma_{\min}^2(\B)} + \epsilon\right)^{2k}||\hat{\vv}_1 - \hat{\vv}_*||^2
,\end{align*}
where $\hat{\vv}_k := (\vv_k^\top\;\vv_{k-1}^\top)^\top$
and $\hat{\vv}_* := (\vv_*^\top\;\vv_*^\top)^\top = \zero$.
Concatenating the results for $\hat{\u}_k$ and $\hat{\vv}_k$ 
concludes the proof.
\qed

\section{Proofs and derivations for Section~\ref{sec:saag}}

\subsection{Proof of Theorem~\ref{thm:sa-mgda,weakmvi,inexact}}
\label{appx:sa-mgda,weakmvi,inexact}
We first extend Theorem~\ref{thm:bppm,weakmvi} 
of the (exact) BPP method
to its inexact variant that approximately computes 
the $h$-resolvent in BPP.
\begin{lemma}
\label{lem:inexactBPP}
Let $\{\x_k\}$ be generated by an inexact BPP,
and $\x_k^*:=\R\x_{k-1}$ 
be an exactly updated point 
from $\x_{k-1}$,
where $\x_k \neq \x_k^*$ in general.
Then, under the conditions in Theorem~\ref{thm:bppm,weakmvi},
the sequence $\{\x_k\}$ satisfies, for $k\ge1$
and for any $\x_*\in X_*^\rho(\M)$,
\begin{align*}
\min_{i=1,\ldots,k}D_h(\x_i^*,\x_{i-1}) \le &\frac{1}{(1-\rho L_h)k}\Bigg(D_h(\x_*,\x_0)\\
&+ \sum_{i=1}^{k}\left(\frac{L_h}{2}||\x_{i}^*-\x_{i}||^2 + L_h||\x_{i}^*-\x_{i}||\cdot||\x_*-\x_{i}^*||\right)\Bigg).
\end{align*}
\end{lemma}
\begin{proof}
	Since $\nabla h(\x_{i-1}) - \nabla h (\x_{i}^*)\in \M\x_{i}^*$, the weak MVI condition implies
	\begin{align*}
	0\le& \langle \nabla h(\x_{i-1}) - \nabla h (\x_{i}^*), \x_{i}^*-\x_*\rangle + \frac{\rho}{2}||\nabla h(\x_{i-1})-\nabla h(\x_{i}^*)||^2\\
	=& D_h(\x_*,\x_{i-1}) - D_h(\x_*,\x_{i}^*) - D_h(\x_{i}^*,\x_{i-1}) + \frac{\rho}{2}||\nabla h(\x_{i-1})-\nabla h(\x_{i}^*)||^2\\
	\le& D_h(\x_*,\x_{i-1}) - D_h(\x_*,\x_{i}^*) - (1-\rho L_h)D_h(\x_{i}^*,\x_{i-1})\\
	=& D_h(\x_*,\x_{i-1}) - D_h(\x_*,\x_{i})\\
	&  + (D_h(\x_*,\x_{i})-D_h(\x_*,\x_{i}^*)) - (1-\rho L_h)D_h(\x_{i}^*,\x_{i-1}).
	\end{align*}
	The term $D_h(\x_*,\x_{i})-D_h(\x_*,\x_{i}^*)$ can be further bounded as
	\begin{align*}
	D_h(\x_*,\x_{i})-D_h(\x_*,\x_{i}^*)
	=& h(\x_{i}^*) - h(\x_{i}) - \langle \nabla h (\x_{i}), \x_*-\x_{i}\rangle
	+\langle \nabla h(\x_{i}^*), \x_*-\x_{i}^*\rangle\\
	=& h(\x_{i}^*) - h(\x_{i}) - \langle \nabla h (\x_{i}), \x_{i}^*-\x_{i}\rangle\\
	&+ \langle \nabla h(\x_{i}^*)-\nabla h (\x_{i}), \x_*-\x_{i}^* \rangle\\
	\le& \frac{L_h}{2}||\x_{i}^*-\x_{i}||^2 + L_h||\x_{i}^*-\x_{i}||\cdot||\x_*-\x_{i}^*||
	.\end{align*}
	Therefore, we get
	\begin{align*}
	(1-\rho L_h)D_h(\x_{i}^*,\x_{i-1}) \le& D_h(\x_*,\x_{i-1}) - D_h(\x_*,\x_{i})\\
	&+ \frac{L_h}{2}||\x_{i}^*-\x_{i}||^2 + L_h||\x_{i}^*-\x_{i}||\cdot||\x_*-\x_{i}^*||.
	\end{align*}
	Then the result follows directly by summing over the inequalities for all $i=1,\ldots,k$ and dividing both sides by $(1-\rho L_h)k$.
\end{proof}

This lemma reduces to Theorem~\ref{thm:bppm,weakmvi}, 
when $\x_k^* = \x_k$
for all $k$.
Under the bounded domain assumption, 
the inequality in Lemma~\ref{lem:inexactBPP} is
further bounded as
\begin{align*}
\min_{i=1,\ldots,k}D_h(\x_i^*,\x_{i-1}) \le \frac{1}{(1-\rho L_h)k}\left(D_h(\x_*,\x_0) + \sum_{i=1}^{k}\frac{3\Omega L_h}{2}||\x_{i}^*-\x_{i}||\right).
\end{align*}

The sequence $\{(\u_i,\vv_i)\}_{i\ge 0}$ of SA-MGDA (with a finite $J$), an instance of the inexact BPP,
satisfies $\u_{i}=\u_{i}^*$,
so $||\x_{i}^*-\x_{i}||=||\vv_{i}^*-\vv_{i}||$.
Since the function except $g$ in the maximization problem with respect to $\vv$ \eqref{eq:uv} is $\left(\frac{1}{\tau}-2L_{\vv\vv}\right)$-strongly concave and $\left(\frac{1}{\tau}+2L_{\vv\vv}\right)$-smooth,
$J$ number of (inner) proximal gradient ascent steps
satisfy
$||\vv_{i}^*-\vv_{i}||\le \Omega\exp{-\frac{\frac{1}{\tau}-2L_{\vv\vv}}{2\left(\frac{1}{\tau}+2L_{\vv\vv}\right)}J}$ (by Theorem~10.29 of \cite{beck:17:fom}) and $L_h = \frac{1}{\tau} + L$,
which concludes the proof.

\subsection{Proof of Theorem~\ref{thm:sa-mgda,strongmvi,inexact}}
\label{appx:sa-mgda,strongmvi,inexact}
We first extend Theorem~\ref{thm:bppm,str,conv} 
of the (exact) BPP method 
to its inexact variant.

\begin{lemma}
\label{lem:inexactBPP,strongmvi}
Let $\{\x_k\}$ be generated by an inexact BPP,
and $\x_k^*:=\R\x_{k-1}$ 
be an exactly updated point 
from $\x_{k-1}$,
where $\x_k \neq \x_k^*$ in general.
Then, under the conditions in Theorem~\ref{thm:bppm,str,conv},
the sequence $\{\x_k\}$ satisfies, for $k\ge1$
and for any $\x_*\in S_*^{\mu}(\M)$,
\begin{align*}
    D_h(\x_*,\x_k)\le &\left(\frac{2\mu}{L_h}+1\right)^{-k} D_h(\x_*,\x_0)\\
    &+ \sum_{i=1}^k\left(\frac{2\mu}{L_h}+1\right)^{-i+1}\left(\frac{L_h}{2}||\x_i^*-\x_i||^2 + L_h||\x_i^*-\x_i||\cdot ||\x_*-\x_i^*||\right).
\end{align*}
\end{lemma}
\begin{proof}
Since $\nabla h(\x_{i-1})-\nabla h (\x_i^*)\in \M\x_i^*$, the strong MVI condition implies
\begin{align*}
    \mu||\x_*-\x_i^*||^2 &\le \langle \nabla h(\x_{i-1})-\nabla h (\x_i^*), \x_i^*-\x_*\rangle\\
    &= D_h(\x_*,\x_{i-1})-D_h(\x_*,\x_i^*)-D_h(\x_i^*,\x_{i-1}).
\end{align*}
Since $D_h(\x_*,\x_i^*)\le \frac{L_h}{2}||\x_*-\x_i^*||^2$, we have
\begin{align*}
    \left(\frac{2\mu}{L_h}+1\right)D_h(\x_*,\x_i^*)\le D_h(\x_*,\x_{i-1}) - D_h(\x_i^*,\x_{i-1})\le D_h(\x_*,\x_{i-1}).
\end{align*}
Therefore,
\begin{align*}
    D_h(\x_*,\x_i)\le& \left(\frac{2\mu}{L_h}+1\right)^{-1} D_h(\x_*,\x_{i-1}) + (D_h(\x_*,\x_i)-D_h(\x_*,\x_i^*))\\
    =& \left(\frac{2\mu}{L_h}+1\right)^{-1} D_h(\x_*,\x_{i-1})\\
    &+ (D_h(\x_i^*,\x_i)-\langle\nabla h(\x_i) - \nabla h(\x_i^*), \x_i^* - \x_*\rangle)\\
    \le& \left(\frac{2\mu}{L_h}+1\right)^{-1} D_h(\x_*,\x_{i-1})\\
    &+ \left(\frac{L_h}{2}||\x_i^*-\x_i||^2 + L_h||\x_i^*-\x_i||\cdot ||\x_*-\x_i^*||\right).
\end{align*}
Then the result follows directly by recursively applying the inequalities for all $i=1,\ldots,k$.
\end{proof}

The 
rest
of the proof is similar to the proof of Theorem~\ref{thm:sa-mgda,weakmvi,inexact} in Appendix~\ref{appx:sa-mgda,weakmvi,inexact}.

\comment{
\subsection{Implicit regularization of the BPP}
\label{appx:reg}

The Yosida approximation of the $h$-resolvent $\R$ is defined as
$\M_h := \nabla h - \nabla h \R$~\cite{borwein:11:aco}. 
By definition, the BPP method on $\M$ 
is equivalent to a mirror descent method~\cite{nemirovski:83}
on $\M_h$,
\begin{align*}
\x_{k+1} = \nabla h^*(\nabla h - \M_h)\x_k
,\end{align*}
where 
$h^*(\y) := \sup_{\x\in\reals^d}\{\inprod{\y}{\x} - h(\x)\}$ 
is the conjugate function of $h$. 
Under Assumption~\ref{assum:M,weak},
the solution sets of $\M$ and $\M_h$ are equivalent
for a $\mu_h$-strongly convex Legendre function $h$ with $\mu_h>\gamma$.
However, $\M_h$ satisfies the following stronger condition,
named $\R$-weak MVI condition,
analogous to the $\R$-monotonicity in~\cite{borwein:11:aco}.
\begin{lemma}
\label{lem:yosida}
Let $h$ be a Legendre function.
The operator $\M$ satisfies Assumption~\ref{assum:weakmvi} 
for some $\rho\ge0$
if and only if 
there exists a solution $\x_* \in X_*(\M)$ such that
\begin{align*}
\inprod{\R\x - \x_*}{\w_h} \ge - \frac{\rho}{2}||\w_h||^2,
\quad \forall(\x,\w_h)\in\gra\M_h
\end{align*}
\end{lemma}

For $h=\frac{L_h}{2}||\cdot||^2$,
the inequality reduces to
$ 
\inprod{\x - \x_*}{\w_h} \ge 
\big(\frac{1}{L_h} - \frac{\rho}{2}\big)||\w_h||^2
$,
which is stronger than the MVI condition if $\rho L_h < 2$.
For the SA-MGDA,
consider a point $\x$ near $\x_*$,
where the function is well approximated by a quadratic
$\phi(\x) \approx \phi(\x_*) 
+ \frac{1}{2}\inprod{\x-\x_*}{\nabla^2 \phi(\x_*)(\x-\x_*)}$.
Then, 
the inequality becomes (see Appendix~\ref{appx:approx})
\begin{align*}
\inprod{\x - \x_*}{\w_h} 
\gtrapprox
\Inprod{\w_h}{\Big(\big(\frac{1}{\tau}\I-\nabla^2\phi(\x_*)\big)^{-1} 
\!\!\!\!- \frac{\rho}{2}\I\Big)\w_h}
,\end{align*}
where the regularization implicitly exploits the Hessian information
of the problem.
This condition
is stronger than the MVI condition if $\frac{1}{\tau}>L$ and
$\rho\big(\frac{1}{\tau}+L\big)<2$.

\subsection{ 
	Approximation of the ${\bf R}$-weak MVI condition for the SA-MGDA}
\label{appx:approx}

For a point $\x$ near $\x_*$,
where the function is well approximated by a quadratic
$\phi(\x) \approx \phi(\x_*) + \frac{1}{2}\inprod{\x - \x_*}{\nabla^2\phi(\x_*)(\x - \x_*)}$,
we have the following approximations for $h$~\eqref{eq:h,saag}:
\begin{align*}
h(\x) &= \frac{1}{2\tau}||\x||^2 - \phi(\x)
\approx \frac{1}{2\tau}||\x||^2 - \phi(\x_*) 
- \frac{1}{2}\inprod{\x - \x_*}{\nabla^2\phi(\x_*)(\x - \x_*)}, \\
\nabla h(\x) &\approx \frac{1}{\tau}\x - \nabla^2\phi(\x_*)(\x - \x_*).
\end{align*}
Then, for any $(\x,\w_h)\in\gra\M_h$,
the resolvent $\R$ satisfies
\begin{align*}
& \nabla h(\R\x) \approx \nabla h(\x) - \w_h, \\
\Leftrightarrow \quad & \frac{1}{\tau}\R\x - \nabla^2\phi(\x_*)(\R\x - \x_*)
\approx \frac{1}{\tau}\x - \nabla^2\phi(\x_*)(\x - \x_*) - \w_h, \\
\Leftrightarrow \quad & \R\x \approx \x 
- \left(\frac{1}{\tau}\I - \nabla^2\phi(\x_*)\right)^{-1}\!\!\!\w_h
,\end{align*}
since $\frac{1}{\tau}\I - \nabla^2\phi(\x_*)$ is positive definite
for $\frac{1}{\tau}>L$.
Therefore, the inequality in Lemma~\ref{lem:yosida} 
can be approximated as
\begin{align*}
\Inprod{\x - \left(\frac{1}{\tau}\I - \nabla^2\phi(\x_*)\right)^{-1}\!\!\!\w_h - \x_*}{\w_h} \gtrapprox - \frac{\rho}{2}||\w_h||^2
.\end{align*}
}


\section{Proof of Theorem~\ref{thm:sa-mgda,line,weakmvi}}
\label{appx:sa-mgda,line,weakmvi}
We first show that $\tau_k$ is lower bounded by $ \frac{\delta}{L+\hat{L}}$ for all $k\ge 0$. Suppose that $\tau_k \ge \frac{\delta}{L+\hat{L}}$ and 
$\tau_{k+1} := \delta^{i_k}\tau_k < \frac{\delta}{L+\hat{L}}$ for some $k\ge 0$. 
Let $\tau_{k+1}':= \delta^{i_k-1}\tau_k = \frac{\tau_{k+1}}{\delta}<\frac{1}{L+\hat{L}}$
and define a function $h_{k+1}':=\frac{1}{2\tau_{k+1}'}\|\cdot\|^2-\phi$,
which is convex and $\Big(\frac{1}{\tau_{k+1}'} + L\Big)$-smooth.
Since $\frac{1}{\tau_{k+1}'}-L>\hat{L}$,
$\x_{k+1}':=\R_{\M}^{h_{k+1}'}(\x_k)$ 
exists, by Lemma~\ref{lem:single}.
In addition, 
%
since $\frac{1}{\tau_{k+1}'}+L < \frac{2}{\tau_{k+1}'}$, the inequality 
	$ 
	\frac{\tau_{k+1}'}{4}||\nabla h_{k+1}'(\x_{k+1}') - \nabla h_{k+1}'(\x_k)||^2
	\le D_{h_{k+1}'}(\x_{k+1}',\x_k)
	$ 
	holds.
This contradicts to the definition of $i_k$. Therefore, $\tau_{k+1}\ge \frac{\delta}{L+\hat{L}}$ for all $k\ge 0$.

Let $\tilde{L} := \frac{2(L + \hat{L})}{\delta}$.
By the definition of $\x_{i+1}$ and the proof of Lemma~\ref{lem:expansive,weakmvi},  
we have
\begin{align*}
0 &\le -D_{h_{i+1}}(\x_{i+1},\x_i) + D_{h_{i+1}}(\x_*,\x_i) - D_{h_{i+1}}(\x_*,\x_{i+1}) \\
&\quad + \frac{\rho}{2}||\nabla h_{i+1}(\x_i) - \nabla h_{i+1}(\x_{i+1})||^2 \nonumber\\
&\le D_{h_{i+1}}(\x_*,\x_i) - D_{h_{i+1}}(\x_*,\x_{i+1}) 
- (1-\rho\tilde{L})D_{h_{i+1}}(\x_{i+1},\x_i) 
\nonumber 
\end{align*}
for all $i\ge 0$.

By summing over the above inequality,
we get
\begingroup
\allowdisplaybreaks
\begin{align*}
&\quad \sum_{i=1}^k (1- 
\rho\tilde{L})
D_{h_i}(\x_i,\x_{i-1})  \\
& \le \sum_{i=1}^k \left(D_{h_i}(\x_*,\x_{i-1}) - D_{h_i}(\x_*,\x_i) \right)\\
&= D_{h_1}(\x_*,\x_0) - D_{h_k}(\x_*,\x_k) +  \sum_{i=1}^{k-1} \left(D_{h_{i+1}}(\x_*,\x_i) - D_{h_i}(\x_*,\x_i) \right) \\
&= D_{h_1}(\x_*,\x_0) - D_{h_k}(\x_*,\x_k) +  \sum_{i=1}^{k-1} \left(\frac{1}{2\tau_{i+1}}-\frac{1}{2\tau_{i}}\right)\|\x_*-\x_i\|^2 \\
&\le D_{h_1}(\x_*,\x_0) - D_{h_k}(\x_*,\x_k) +  \sum_{i=1}^{k-1} \left(\frac{1}{2\tau_{i+1}}-\frac{1}{2\tau_{i}}\right) C^2 \\
&\le D_{h_1}(\x_*,\x_0) - D_{h_k}(\x_*,\x_k) +  \frac{1}{2\tau_k} C^2 \\
&\le D_{h_1}(\x_*,\x_0) - D_{h_k}(\x_*,\x_k) +  \frac{\tilde{L}}{4}C^2\\ 
&\le D_{h_1}(\x_*,\x_0) +  D_{\phi}(\x_*,\x_k) +  \frac{\tilde{L}}{4}C^2 \\
&\le D_{h_1}(\x_*,\x_0) + \frac{L}{2}C^2 + \frac{\tilde{L}}{4}C^2 \\
&\le D_{h_1}(\x_*,\x_0) + \frac{\tilde{L}}{2}C^2 
,\end{align*}
\endgroup
which uses $D_{h_k}(\x_*,\x_k) = \frac{1}{2\tau_k}\|\x_*-\x_k\|^2 - D_{\phi}(\x_*,\x_k)$.
Therefore, the result in Theorem~\ref{thm:sa-mgda,line,weakmvi} directly 
follows 
under
the condition that
$\rho\tilde{L}<1$, which is equivalent to
$\delta > 2\rho(L+\hat{L})$.
We also need $\rho< \frac{1}{2(L+\hat{L})}$, so that 
$\delta\in\big(2\rho(L+\hat{L}),1\big)$ exists.
\qed


\section{Details of the structure of the neural network}
\label{appx:network}

\begin{table*}[h!]
	\centering
	\begin{tabular}{l l}
		\toprule[1pt] 
		Layer Type & Shape \\
		\midrule[1pt] 
		Convolution $+$ tanh & 3 $\times$ 3 $\times$ 5 \\
		Max Pooling & 2 $\times$ 2 \\
		Convolution & 3 $\times$ 3 $\times$ 10 \\
		Max Pooling & 2 $\times$ 2 \\
		Fully Connected $+$ tanh & 250 \\
		Fully Connected $+$ tanh & 100 \\
		Softmax & 3 \\
		\bottomrule[1pt] 
	\end{tabular}
	\caption{Details of the Structure of the Neural Network}
\end{table*}

\section*{Acknowledgments}
We thank the anonymous referees for their careful reading
and their useful comments and suggestions.

\bibliographystyle{siamplain}
\bibliography{master,mastersub}

\end{document}